\documentclass[11pt]{amsart}
\usepackage{mathrsfs}
\usepackage{amssymb}
\usepackage{hyperref}
\usepackage{enumitem}
\usepackage{color,soul}
\usepackage{geometry}
\usepackage{cite}

\newtheorem{theorem}{Theorem}[section]
\newtheorem{lemma}[theorem]{Lemma}
\newtheorem{proposition}[theorem]{Proposition}
\newtheorem{corollary}[theorem]{Corollary}

\theoremstyle{definition}

\theoremstyle{remark}

\title[Information-Theoretic Isoperimetric Inequalities]{An Information-Theoretic Route to Isoperimetric Inequalities via Heat Flow and Entropy Dissipation}

\author{Amandip Sangha} 
\thanks{Preprint}
\address{The Climate and Environmental Research Institute NILU}
\email{asan@nilu.no}

\begin{document}

\begin{abstract}
We develop an information-theoretic approach to isoperimetric inequalities based on entropy dissipation under heat flow.
By viewing diffusion as a noisy information channel, we measure how mutual information about set membership decays over time.
This decay rate is shown to be determined by the boundary measure of the set, leading to a new proof of the Euclidean isoperimetric inequality with its sharp constant.
The method extends to Riemannian manifolds satisfying curvature–dimension conditions, yielding Lévy–Gromov and Gaussian isoperimetric results within a single analytic principle.
Quantitative and stability bounds follow from refined entropy inequalities linking information loss to geometric rigidity.
The approach connects geometric analysis and information theory, revealing how entropy dissipation encodes the geometry of diffusion and boundary.
\end{abstract}

\maketitle
\setcounter{tocdepth}{1} 
\tableofcontents

\section{Introduction}

Isoperimetric inequalities lie at the foundation of geometric analysis: among
all sets of fixed volume in $\mathbb{R}^n$, the Euclidean ball uniquely minimizes
surface area \cite{Maggi2012}. On Riemannian manifolds, curvature modifies this principle and
leads to sharp comparison results such as the Lévy--Gromov inequality
\cite{Gromov1980,GromovIsoPer,MorganJohnson2000}. A large body of work—from classical symmetrization
to modern diffusion and optimal-transport techniques—has established these results through
geometric measure theory \cite{federer1969geometric}, variational calculus \cite{Giusti1984}, and $\Gamma$–calculus \cite{BakryGentilLedoux2014}.

This paper develops a new analytic route to isoperimetric inequalities based on
\emph{entropy dissipation under heat flow}.  The central observation is that heat
diffusion acts as a noisy communication channel that gradually destroys
information about which side of a boundary a point lies on.  The rate at which
this mutual information decays is controlled precisely by the measure of the
boundary, linking information loss directly to perimeter.  This perspective
unites concepts from information theory—entropy, mutual information, and
data–processing inequalities—with classical tools of geometric analysis such as
the heat semigroup and small–time heat–kernel asymptotics.

From this viewpoint we derive, in the Euclidean case, a quantitative identity
between the short–time increase of conditional entropy and the perimeter
measure, yielding a new proof of the sharp isoperimetric inequality and rigidity
for balls.  The method extends naturally to manifolds satisfying
curvature–dimension conditions $\mathrm{CD}(K,n)$, where entropy dissipation
estimates along the heat flow give rise to curvature–dependent comparison
theorems of Lévy–Gromov type.  Refined versions lead to quantitative and
stability results: near–optimal information decay implies geometric closeness to
the extremal ball or spherical cap.  The approach offers a unified analytic
principle encompassing Euclidean, Gaussian, and positively or negatively curved
spaces, and connects the monotonicity of mutual information under diffusion with
geometric rigidity and concentration phenomena.

\medskip
\noindent
\textbf{Main curvature–corrected entropy law.}
A central outcome of the theory is the \emph{curvature–corrected
entropy expansion}, established in Theorem~\ref{thm:global-expansion}:
\begin{equation}\label{eq:curvature_expansion_intro}
H_E(t)
\;=\;
C\,\sqrt{t}\,\mathrm{Per}(E)
\;-\;
C\,K\,t^{3/2}\,\mathrm{Per}(E)
\;+\;o(t^{3/2}),
\qquad t \downarrow 0,
\end{equation}
valid on every Riemannian manifold with $\mathrm{Ric}\ge K g$.
This formula refines the Euclidean entropy–perimeter law by revealing how Ricci
curvature enters as the \emph{first correction term} to the information–dissipation rate.
It provides a quantitative bridge between entropy production, mean curvature, and
ambient Ricci curvature, placing classical isoperimetric inequalities and their stability
within a single analytic expansion.

\medskip
\noindent
\textbf{Universal constant.}
The coefficient $C$ in \eqref{eq:curvature_expansion_intro} is
\emph{universal}—it arises from a simple one-dimensional Gaussian boundary–layer integral that
is independent of the ambient dimension, geometry, and topology.
Its dimension–free nature underscores that the entropy–perimeter correspondence is intrinsic
to diffusion itself rather than to any specific manifold.  In this sense,
Theorem~\ref{thm:global-expansion} provides a single analytic expansion — the curvature-corrected entropy–perimeter law — whose specializations to flat, positively curved, negatively curved, and weighted (Gaussian) settings correspond to the Euclidean, spherical, hyperbolic, and Gaussian isoperimetric inequalities, respectively.

\medskip
\noindent
\textbf{Structural correspondence.}
This work reveals a structural bridge between three central analytic themes:
the short–time asymptotics of the heat kernel, the geometric measure theory of
sets of finite perimeter, and the information–theoretic notion of entropy.
By viewing surface area as the leading coefficient in the entropy expansion of
heat diffusion, we obtain a precise analytic correspondence between diffusion
dynamics and boundary geometry.  Within this framework, the classical
isoperimetric problem appears as the variational principle governing minimal
initial entropy production, while curvature enters naturally through a
higher–order term in the expansion.  The resulting formulation places
isoperimetric geometry, curvature, and entropy within a single analytic
structure and shows how isoperimetric inequalities emerge from the thermodynamic
behavior of diffusion, with entropy and curvature acting as analytic mediators
between geometry and heat flow.

\medskip
\noindent
\textbf{A new perspective.}
Traditional proofs of the isoperimetric inequality have followed two dominant paradigms. 
The \emph{optimal-transport approach} derives isoperimetry from the displacement convexity of entropy along Wasserstein geodesics 
\cite{CEMS,OttoVillani,Milman,CavallettiMondino2017}, 
while the \emph{concentration and functional-inequality approach} 
traces the result to \\ curvature–dimension conditions, Log-Sobolev inequalities, or measure concentration phenomena 
\cite{BakryEmery1985,BobkovLedoux,Ledoux}. 
Both routes are powerful but geometrically or probabilistically indirect, 
relying on the global structure of transport maps or on integrated concentration bounds.

The present work introduces an \emph{information-theoretic route} that departs sharply from these frameworks. 
We interpret the heat flow itself as an information channel, viewing the diffusion of a set as the gradual loss of information about its set-membership variable. 
The decay of mutual information between the initial label and its diffused image quantifies the entropy dissipation of this binary process, 
and its short-time expansion yields the perimeter as the leading-order information-production term. 
This formulation identifies the isoperimetric inequality as an \emph{information dissipation principle}. Curvature bounds enter naturally through the Bochner identity, and equality corresponds to the rigidity of the Fisher information flow. 
The approach thus provides a direct analytic route—avoiding transport maps and concentration estimates—yet recovers the sharp Euclidean and L\'evy–Gromov constants with full rigidity.

\medskip
\noindent
\textbf{Outline of the paper.}
The paper is organized as follows. Section 2 introduces the necessary notation and background from geometric measure theory, heat semigroups, and information theory. Section 3 formulates the heat flow as an information channel, establishing the monotonicity properties of mutual information. Section 4 discusses entropy dissipation and its link to boundary geometry, providing the proof of the Euclidean isoperimetric inequality. Section 5 focuses on the derivation of quantitative stability results. Section 6 extends the method to manifolds under the curvature–dimension condition $CD(K,n)$. Section 7 presents the curvature-corrected entropy expansion and its geometric interpretation. Section 8 provides examples that illustrate the results in various geometric settings. The final section concludes with a summary of results and their implications.

\section{Preliminaries}
We recall the basic notions and notation used throughout the paper. 

\subsection{Geometric measure and perimeter}
Let $(M,g)$ be an $n$–dimensional compact boundaryless Riemannian manifold with metric $g$ and volume measure $d\mu = dvol_g$. 

For a measurable set $E\subset M$, the \emph{perimeter}
\cite{DeGiorgi1954,AmbrosioFuscoPallara2000}
is defined as the total variation of its indicator function~$\mathbf{1}_E$
with respect to~$d\mu$, which can be expressed equivalently as
\[
Per(E)
:= \sup\Bigl\{
   \int_E \mathrm{div}_g X\,d\mu :
   X\in C_c^1(TM),\ |X|\le 1
   \Bigr\}.
\]
A set $E$ has \emph{finite perimeter} if $Per(E)<\infty$.
When $E$ has smooth boundary,
\[
Per(E)=H^{n-1}_g(\partial E),
\qquad
\text{and in general } Per(E)=H^{n-1}_g(\partial^*E),
\]
where $\partial^*E$ is the reduced boundary in the sense of
De~Giorgi~\cite{DeGiorgi1954}. 

Let $H^{n-1}_g$ denote the $(n\!-\!1)$–dimensional Hausdorff
measure induced by the Riemannian metric~$g$.  For a smooth hypersurface
$\Sigma \subset M$, which is by definition a smooth $(n-1)$–dimensional embedded submanifold of $M$ locally given as the zero level set of a smooth
function with nonvanishing gradient, then the Hausdorff measure $H^{n-1}_g$ coincides with the $(n\!-\!1)$–dimensional surface area
obtained by integrating the \emph{induced surface element}
$d\sigma_g$ defined as the pullback of the interior product of the Riemannian
volume form with the unit normal vector field:
\[
  d\sigma_g := \iota^{*}\bigl(i_{\nu}\, d\mu\bigr),
\]
where $\iota:\Sigma \hookrightarrow M$ is the inclusion map, $\nu$ is a
locally defined unit normal field along~$\Sigma$, and $i_\nu$ is the interior product (contraction) of a differential form with the unit normal vector field~$\nu$. Now the corresponding
$(n-1)$--dimensional Riemannian Hausdorff measure $H^{n-1}_g$ coincides with
the induced surface element, so that
\[
  H^{n-1}_g(\Sigma)
    = \int_{\Sigma} d\sigma_g.
\]
If a measurable set $E$ has smooth boundary, then \cite{AmbrosioFuscoPallara2000}
\[ Per(E)=H^{n-1}_g(\partial E) = \int_{\partial E} d\sigma_g. \]  
The \emph{isoperimetric profile} \cite{Bayle2004, MorganGMT, Milman2010} of $(M,g)$ is $I_M: [0,1] \longrightarrow [0,\infty)$,
\[
I_M(s)=\inf\{Per(E):\mu(E)=s\},\qquad s\in[0,1].
\]

We use the notation \(A \lesssim B\) to indicate that \(A \leq C\,B\) for some
constant \(C > 0\) independent of the relevant parameters. Recall that the volume of the unit $n$-ball in $\mathbb{R}^n$ is 
$\omega_n = \pi^{n/2}/\Gamma(\tfrac{n}{2}+1)$, so for the $n-1$-sphere 
$|S^{\,n-1}| = n\,\omega_n = 2\pi^{n/2}/\Gamma(\tfrac{n}{2})$.

\subsection{Heat semigroup and diffusion}\label{sec:heatsemigroupdiffusion}
The \emph{heat semigroup} $(P_t)_{t\ge0}$ associated with the Laplace--Beltrami
operator $\Delta_g$ is defined \cite{GrigorYan2009} for functions $f\in L^{\infty}(M,\mathbb{R})$ by
\[
P_t f(x) = \int_M K_t(x,y)\,f(y)\,d\mu(y),
\]
where $K_t(x,y)$ is the heat kernel satisfying
\[
\partial_t K_t(x,y) = \Delta_g K_t(x,y), \qquad
\lim_{t\downarrow0} K_t(x,y) = \delta_x(y).
\]
Equivalently, $(P_t f)$ is the unique smooth solution to the heat equation
$\partial_t u = \Delta_g u$ with initial condition $u(\cdot,0)=f$. The heat semigroup is defined as
\[
P_t = e^{-t \Delta},
\]
and the kernel $K_t(x,y)$ has the following asymptotic form as \(t \to 0\):
\[
K_t(x,y) = \frac{1}{(4\pi t)^{n/2}} e^{-\frac{d^2(x,y)}{4t}},
\]
where \(d(x,y)\) is the distance between \(x\) and \(y\) on the manifold \(M\). For a set $E \subset M$ with smooth boundary $\partial E$, the entropy functional $H_E(t)$ is defined as:
\[
H_E(t) = \int_{\partial E} P_t \mathbf{1}_E(x) \log\left( P_t \mathbf{1}_E(x) \right) \, d\sigma(x).
\]

On a Riemannian manifold $(M,g)$ with Laplace--Beltrami operator
$L=\Delta_g$ and heat semigroup $(P_t)_{t\ge0}$, we say that
$(M,g)$ satisfies the \emph{curvature–dimension condition}
$\mathrm{CD}(K,n)$, for parameters $K\in\mathbb{R}$ and
$n\in[1,\infty]$, if its Ricci curvature tensor satisfies
\[
\mathrm{Ric}_g(v,v) \;\ge\; K\,|v|_g^2
\quad \text{for all } v\in TM,
\qquad \text{and}\quad \dim M \le n.
\]
Equivalently, the diffusion generator $L=\Delta_g$ obeys the
Bakry--Émery inequality
\[
\Gamma_2(f) \;\ge\; K\,\Gamma(f)
  + \tfrac{1}{n}\,(Lf)^2,
\qquad \forall f\in C^\infty(M),
\]
where $\Gamma(f)=|\nabla f|^2$ and
$\Gamma_2(f)
= \tfrac12\,L|\nabla f|^2 - \langle\nabla f,\nabla Lf\rangle$.
This is the analytic formulation of the curvature–dimension condition
of Bakry and Émery \cite{BakryEmery1985,BakryGentilLedoux2014},
and is equivalent to the geometric lower Ricci bound used by
Lott, Sturm and Villani
\cite{Sturm2006a,Sturm2006b,LottVillani2009}.
Under the $CD(K,n)$ condition, the heat semigroup $(P_t)$ enjoys
gradient and entropy contraction properties crucial for functional
inequalities.

\subsection{Entropy and mutual information}
For a discrete random variable $L$ taking values in $\{0,1\}$ with $\mathbb P(L=1)=p$, the Shannon entropy \cite{Shannon1948} is
$H(L)=-p\log p-(1-p)\log(1-p)$. In information theory, entropy quantifies the uncertainty or information content of a random variable. In physics and stochastic analysis, it measures disorder or the spread of probability—that is, how “mixed” or “uncertain” a system’s state is. Mathematically, both notions correspond to the same functional, interpreted differently: one informational, the other thermodynamic.

If $Y$ is another random variable on $M$, the conditional entropy and mutual information are
\[
H(L|Y)=\mathbb E_{Y}[H(L|Y=y)] = \mathbb E[h(\mathbb P(L=1|Y))], \qquad
I(L;Y)=H(L)-H(L|Y),
\]
where $h(u)=-u\log u-(1-u)\log(1-u)$ is the well known \emph{binary entropy function} \cite{Shannon1948} that represents the Shannon entropy of a Bernoulli random variable with success probability $u$, used to define $H(L)$ above.

Two basic inequalities will be used repeatedly:\\
\noindent
\emph{Data processing inequality.} \cite{CoverThomas2006} If $L\to X\to Y$ is a Markov chain, then $I(L;Y)\le I(L;X)$.

\noindent
\emph{Fano's inequality.} \cite{Fano1961,CoverThomas2006} Let $\mathcal{L}$ denote the finite set of possible values (alphabet) of $L$, and let $|\mathcal{L}|$ be its cardinality. The minimal classification error $P_e$ in estimating $L$ from $Y$ satisfies
  $H(L|Y)\le h(P_e)+P_e\log(|\mathcal L|-1)$.

For a measurable set $E\subset M$, we will consider 
\[
L = \mathbf{1}_E(X)
=
\begin{cases}
1, & \text{if } X\in E,\\[4pt]
0, & \text{otherwise.}
\end{cases}
\]
for $X\sim\mu$, i.e. the random variable $X$ is a point randomly sampled from the manifold $M$ according to the probability measure (Riemannian volume element) $\mu$. Thus, $L$ is a binary random variable (label) indicating whether the random point $X$ lies inside the set $E\subset M$. Hence, $L$ takes values in $\mathcal{L}=\{0,1\}$, with
$|\mathcal{L}|=2$, and Fano’s inequality reduces to $H(L|Y)\le h(P_e)$.

\paragraph{Communication, entropy, and thermodynamics.}
In information theory, a \emph{noisy channel} refers \cite{CoverThomas2006} to a conditional distribution
$P_{Y|X}$ that maps each input symbol $X$ to a random output $Y$.
The pair $(X,Y)$ thereby models transmission through a stochastic medium:
the input $X$ is “sent”, and the receiver observes a possibly corrupted
version $Y$ drawn from $P_{Y|X}$. The amount of information preserved
between input and output is quantified by the \emph{mutual information}
$I(X;Y)$ introduced above.

From a thermodynamics viewpoint, entropy plays an analogous role: it measures the
degree of microscopic uncertainty compatible with a macroscopic description.
In statistical mechanics, many microstates correspond to the same observable
macrostate, and entropy quantifies this hidden multiplicity \cite{ReifStatMech}.  As a system evolves irreversibly—such as under heat diffusion or mixing—microscopic
information is gradually lost, and entropy increases.  This correspondence
highlights the deep connection between information and disorder: both describe
the loss of distinguishability among states as noise or thermal agitation
spreads through a system.

Throughout this paper we interpret the heat semigroup $(P_t)_{t\ge0}$ as
inducing a continuous-time noisy channel as follows. Each $P_t$ admits an integral representation
\(
P_t f(x)=\int_M K_t(x,y)\,f(y)\,d\mu(y),
\)
where $K_t(x,\cdot)$ is the heat kernel, viewed as a probability measure on~$M$.
Thus, for fixed $t>0$, the mapping
\(
x \mapsto K_t(x,\cdot)
\)
is a Markov kernel, which we interpret as the (discrete-time) noisy channel
\[
P_{Y_t|X=x}(B) := \int_B K_t(x,y)\,d\mu(y),
\]
where $B\subset M$ is a measurable set.
The family $\{K_t\}_{t>0}$ is a \emph{continuous-time family of channels}
satisfying the composition (semigroup) law
\[
K_{t+s}(x,\cdot)=\int_M K_t(z,\cdot)\,K_s(x,dz),
\qquad
\text{equivalently }~ P_{Y_{t+s}|X} = P_{Y_t|Y_s}\!\circ P_{Y_s|X}.
\]
We will therefore speak of “the heat channel at time $t$”, meaning the
Markov kernel $K_t(\cdot,\cdot)$ induced by the heat semigroup.

We define the universal constant $C$ that appears throughout the entropy expansions,
\begin{equation}\label{eq:constantC}
\begin{aligned}
C &:= \sqrt{2}\int_{\mathbb{R}} h(\Phi(u))\,du \\
  &= \sqrt{2}\int_{\mathbb{R}}\!\Big(-\Phi(u)\log \Phi(u) 
  - (1-\Phi(u))\log(1-\Phi(u))\Big)\,du,
\end{aligned}
\end{equation}
where \[
\Phi(a)
:= \frac{1}{\sqrt{2\pi}}
   \int_{-\infty}^{a} e^{-t^2/2}\,dt
\]
is the Gaussian CDF. The constant \(C\) governs the leading-order term in the Euclidean entropy–perimeter relation.

\section{The Heat Flow as an Information Channel}
We now interpret the heat semigroup $(P_t)_{t\ge0}$ on $(M,g)$ as an
information channel that gradually destroys knowledge about the
membership of a point in a given set $E\subset M$.

\subsection{The noisy--label model}
Let $X$ be a random point distributed according to the normalized volume
measure $\mu$ on $M$, and let $L=\mathbf 1_E(X)\in\{0,1\}$ denote the indicator or 
\emph{label} indicating whether $X$ lies inside or outside $E$.  
For $t>0$ we define the \emph{diffused observation}
$Y_t\sim K_t(X,\cdot)$, where $K_t$ is the heat kernel associated with
$\Delta_g$.
Thus $X\mapsto Y_t$ is a Markov kernel describing the random location of
a Brownian particle after time~$t$, and
\[
  L \;\longrightarrow\; X \;\longrightarrow\; Y_t
\]
forms a Markov chain (Section $\ref{subsec:mutual-info}$).
The map $X\mapsto Y_t$ therefore plays the role of a \emph{noisy
channel} through which the binary label~$L$ is transmitted with noise
generated by heat diffusion. 

\subsection{Mutual information under diffusion}\label{subsec:mutual-info}
Define the conditional probability
\[
p_t(x)=\mathbb P(L=1\mid Y_t=x)=P_t\mathbf 1_E(x),
\]
which satisfies the heat equation $\partial_t p_t=\Delta_g p_t$ with
initial data $\mathbf 1_E$ \cite{GrigorYan2009}.  
The conditional entropy of $L$ given the observation $Y_t$ is
\[
H(L|Y_t)=\int_M h(p_t(x))\,d\mu(x),
\qquad
h(u)=-u\log u-(1-u)\log(1-u),
\]
and the mutual information is
$I(L;Y_t)=H(L)-H(L|Y_t)$.  
Now, $L\to X\to Y_t$ is a Markov chain \cite{CoverThomas2006}. Indeed, as $L = \mathbf{1}_E(X)$ is $X$-measurable and $Y_t$ is sampled from the heat kernel $K_t(X,\cdot)$, we have for every measurable set $B \subset M$,
\[
\mathbb{P}(Y_t \in B \mid X, L)
= K_t(X, B)
= \mathbb{P}(Y_t \in B \mid X),
\]
which shows that the conditional distribution of $Y_t$ given $(X,L)$ does not depend on $L$; hence $L \to X \to Y_t$ is a Markov chain. The \emph{data--processing inequality} implies that $I(L;Y_t)$ is
nonincreasing in~$t$ \cite{Shannon1948, CoverThomas2006, Stam1959, Costa1985}, reflecting the irreversible loss of information
under diffusion. Indeed, by the Markov property of the heat semigroup, one has the Markov chain
\[
L \;\longrightarrow\; Y_s \;\longrightarrow\; Y_t,
\qquad 0 < s < t,
\]
where $Y_t$ is obtained from $Y_s$ by an additional diffusion step.
The chain rule for mutual information gives
\[
I(L;Y_s,Y_t)
= I(L;Y_t) + I(L;Y_s\mid Y_t)
= I(L;Y_s) + I(L;Y_t\mid Y_s).
\]
Since $L\to Y_s\to Y_t$ implies $I(L;Y_t\mid Y_s)=0$, it follows that
\[
I(L;Y_t)
= I(L;Y_s) - I(L;Y_s\mid Y_t)
\le I(L;Y_s),
\]
and the lost information is precisely $I(L;Y_s\mid Y_t)\ge0$.
This expresses the fact that once information about $L$ is lost through
diffusion, it cannot be recovered by further evolution—
hence, the process is irreversible.

The quantities $H(L|Y_t)$ and $I(L;Y_t)$ measure the uncertainty and information loss about the membership of $X$ in $E$ after diffusion time $t$. We next quantify the rate of this information loss.

For a measurable set \(E \subset M\), we define the associated entropy functional
\[
H_E(t) := \int_M h(p_t(x))\,d\mu(x)
= \int_M h\!\bigl(P_t\mathbf{1}_E(x)\bigr)\,d\mu(x).
\]
This quantity coincides with the conditional entropy \(H(L \mid Y_t)\)
introduced above, since \(h(p_t(x))\) represents the pointwise entropy of the
posterior probability \(P(L=1 \mid Y_t=x)\).

\subsection{Information dissipation identity}
The decay of mutual information along the heat flow can be expressed in
differential form, relating the rate of information loss to the geometric
gradient structure of the diffusion.  The following proposition makes this
relation precise.

\begin{proposition}\label{prop:inf-diss-id}
We have the information dissipation identity
\begin{equation}\label{eq:dissipationIdentity}
I_t := \frac{d}{dt}H(L|Y_t)
  = -\int_M \frac{|\nabla p_t|^2}{p_t(1-p_t)}\,d\mu.
\end{equation}
\end{proposition}
\begin{proof}
Write $H(t):=H(L\mid Y_t)=\int_M h(p_t)\,d\mu$, where $p_t:=P_t\mathbf{1}_E$. For $t>0$, $p_t\in C^\infty(M)$ solves
$\partial_t p_t=\Delta_g p_t$, and $0<p_t<1$. Then
\begin{align*}
\frac{d}{dt}H(t)
&= \int_M h'(p_t)\,\partial_t p_t\,d\mu \\
&= \int_M h'(p_t)\,\Delta_g p_t\,d\mu \\
&= -\int_M \langle \nabla h'(p_t),\nabla p_t\rangle\,d\mu \\
&= -\int_M h''(p_t)\,|\nabla p_t|^2\,d\mu \\
&= -\int_M \frac{|\nabla p_t|^2}{p_t(1-p_t)}\,d\mu
\end{align*}
\end{proof}
The quantity $I_t$ can be interpreted as the Fisher information of the diffused label $p_t$ with respect to the Riemannian measure. The nonnegative quantity $I_t$ measures the instantaneous rate
of information dissipation.
The identity above expresses the monotonic dissipation of information under
heat diffusion, mirroring the de~Bruijn relation \cite{ParkSerpedinQaraqe2012} for continuous densities.
It will serve as the analytic foundation for the entropy–perimeter laws
developed in the following section.

\medskip
\noindent
The remainder of the paper develops this principle quantitatively,
first in Euclidean space and then under general curvature--dimension
conditions, where curvature modifies the rate of entropy dissipation and
thus the shape of the optimal isoperimetric profile as expected.

\section{Entropy Dissipation and Boundary Geometry}
We now establish the connection between the rate of entropy growth
$\,\frac{d}{dt}H(L|Y_t)\,$ and the perimeter of a set $E\subset M$.
The main result identifies the first--order term of the entropy
expansion with the $(n{-}1)$–dimensional Hausdorff measure of
$\partial E$.

\subsection{Small--time expansion of conditional entropy}
Let $p_t=P_t\mathbf 1_E$.  For smooth $\partial E$,
the asymptotic behaviour of the heat semigroup near the boundary is
governed by Varadhan's formula \cite{Varadhan1967} and the Minakshisundaram–Pleijel expansion \cite{MinakshisundaramPleijel1949,Rosenberg1997}:
\[
K_t(x,y)\simeq (4\pi t)^{-n/2}\,
  e^{-d_g(x,y)^2/4t}\,(1+O(t)).
\]

The following result is the main analytic step of our approach.
It quantifies how the heat semigroup smooths indicator functions at short times,
and expresses the resulting $L^1$–jump in terms of the boundary measure.
This asymptotic formula is the key technical step linking entropy growth to perimeter.
\begin{theorem}[Heat-content and $L^1$--jump asymptotics]
\label{thm:heatL1}
Let $(M,g)$ be a smooth compact Riemannian $n$–manifold, and let $E\subset M$ have smooth boundary~$\partial E$.
Let $(P_t)_{t\ge0}$ be the heat semigroup associated to the Laplace--Beltrami operator.
Then, as $t\downarrow0$,
\begin{equation}\label{eq:heatL1}
\|P_t \mathbf{1}_E - \mathbf{1}_E\|_{L^1(M)}
= \frac{2}{\sqrt{\pi}}\;\sqrt{t}\;H^{n-1}_g(\partial E) \;+\; o(\sqrt{t}).
\end{equation}
\end{theorem}

\begin{proof}
We write $K_t(x,y)$ for the heat kernel, so $P_t f(x)=\int_M K_t(x,y)f(y)\,d\mu(y)$.

\medskip
\noindent
As the boundary $\partial E$ is smooth and $M$ is compact, there exists $\rho>0$ such that the normal exponential map
\[
\Psi:\ \partial E \times (-\rho,\rho)\ \longrightarrow\ U_\rho := \{x\in M:\ \mathrm{dist}(x,\partial E)<\rho\},
\Psi(y,s)=\exp_y(s\,\nu(y)),
\]
is a diffeomorphism, where $\nu(y)$ is the outer unit normal to $\partial E$ at $y$.
In the Fermi coordinates $x=\Psi(y,s) = \exp_y(s\,\nu(y))$, the Riemannian measure $\mu$ has the form
\[
d\mu(x)=J(y,s)\,ds\,d\sigma(y),\qquad J(y,s)=1+\kappa(y)s+O(s^2),
\]
for a smooth function $\kappa$ given by $\kappa(y) = -\operatorname{tr}\mathrm{II}_y$, where $\mathrm{II}_y$ is the second fundamental form of $\partial E \subset M$ \cite{Rosenberg1997}.
Let $\chi\in C_c^\infty(\mathbb{R})$ be a smooth cutoff function with $\chi\equiv1$ on $[-1,1]$ and $\mathrm{supp}\,\chi\subset (-2,2)$.
Fix $\alpha\in(0,1/2)$ and put $R_t:=t^\alpha$. We split
\[
\|P_t\mathbf{1}_E-\mathbf{1}_E\|_{L^1(M)} \;=\; I_t^{\mathrm{near}} + I_t^{\mathrm{far}},
\]
where
\begin{align*}
I_t^{\mathrm{near}}
&:= \int_{U_\rho} |P_t\mathbf{1}_E(x)-\mathbf{1}_E(x)|\,d\mu(x) \\
&= \int_{\partial E}\!\!\int_{-\,\rho}^{\rho} |P_t\mathbf{1}_E(\Psi(y,s))-\mathbf{1}_{\{s>0\}}|\,J(y,s)\,ds\,d\sigma(y),
\end{align*}
and $I_t^{\mathrm{far}}$ is the integral over $M\setminus U_\rho$.

Since $\mathrm{dist}(x,\partial E)\ge \rho$ for $x\in M\setminus U_\rho$,
Gaussian off--diagonal bounds for the heat kernel
\cite{GrigorYan2009} yield
$K_t(x,y)\le C t^{-n/2} e^{-d_g(x,y)^2/5t}$.
Hence,
\[ 
|P_t\mathbf{1}_E(x)-\mathbf{1}_E(x)|\le C t^{-n/2} e^{-\rho^2/5t},
\]
and integrating over the compact set $M\setminus U_\rho$ gives
$I_t^{\mathrm{far}} = O(e^{-c/t})$, and since $e^{-c/t} = o(\sqrt{t})$ as $t\downarrow0$,
we conclude $I_t^{\mathrm{far}} = o(\sqrt{t})$ as well.
Within $U_\rho$, we further split using $\chi(s/R_t)$ into a boundary layer term and a remainder term:
\[
I_t^{\mathrm{near}} = I_t^{\mathrm{bl}} + I_t^{\mathrm{rem}},
\quad
I_t^{\mathrm{bl}}
:= \int_{\partial E}\!\!\int_{\mathbb{R}}
\big|P_t\mathbf{1}_E(\Psi(y,s))-\mathbf{1}_{\{s>0\}}\big|\,
\chi\!\left(\tfrac{s}{R_t}\right)
J(y,s)\,ds\,d\sigma(y),
\]
where the $s$–integral is over $|s|\le 2R_t$, and
\[
I_t^{\mathrm{rem}}
:= \int_{\partial E}\!\!\int_{R_t<|s|<\rho}
\big|P_t\mathbf{1}_E(\Psi(y,s))-\mathbf{1}_{\{s>0\}}\big|\,
\big(1-\chi(\tfrac{s}{R_t})\big)J(y,s)\,ds\,d\sigma(y).
\]
Since $|s|\ge R_t$ on the support of the integrand of $I_t^{\mathrm{rem}}$ and $R_t/t^{1/2}\to\infty$ (because $\alpha<1/2$),
Gaussian tail bounds for $K_t$ imply $I_t^{\mathrm{rem}}=o(\sqrt t)$. Indeed, using the boundary profile $p_t(\Psi(y,s))\approx \Phi\!\big(s/\sqrt{2t}\big)$ one has
\[
\big|P_t\mathbf 1_E(\Psi(y,s))-\mathbf 1_{\{s>0\}}\big|
= \Phi\!\Big(-\frac{|s|}{\sqrt{2t}}\Big)
\le \frac{\sqrt{t}}{\sqrt{\pi}\,|s|}\,e^{-\,s^2/(4t)}.
\]
Hence, for $0<R_t<t^{1/2}$ and $d\mu=J(y,s)\,ds\,d\sigma(y)$ with $J=1+O(|s|)$,
\[
I_t^{\mathrm{rem}}
\le C\, H^{n-1}_g(\partial E)
\int_{R_t<|s|<\rho}\!\frac{\sqrt{t}}{|s|}\,e^{-\,s^2/(4t)}\,ds
\le C\,\sqrt{t}\int_{R_t/\sqrt{4t}}^{\infty}\frac{e^{-u^2}}{u}\,du
\le C\,\sqrt{t}\,e^{-\,R_t^2/(4t)}.
\]
Since $R_t=t^{\alpha}$ with $\alpha\in(0,1/2)$ implies $R_t^2/t=t^{2\alpha-1}\to\infty$, we obtain
$\,I_t^{\mathrm{rem}}=o(\sqrt{t})$ as $t\downarrow0$.

Thus
\[
\|P_t\mathbf{1}_E-\mathbf{1}_E\|_{L^1(M)} \;=\; I_t^{\mathrm{bl}} + o(\sqrt t).
\]

\medskip
\noindent
Now we fix $y\in\partial E$ and work in Fermi (boundary normal) coordinates $(z,s)$ near $y$, where $z\in\mathbb{R}^{n-1}$ are tangential coordinates on $\partial E$ and $s$ is signed distance. More precisely, we let
\[
\Psi:\partial E\times(-\rho_0,\rho_0)\to M,
\qquad
\Psi(z,s):=\exp_z(s\,\nu(z)),
\]
where $\nu(z)$ denotes the outward unit normal to $\partial E$ at $z$,
and $\rho_0>0$ is chosen so that $\Psi$ is a diffeomorphism onto its image.
In these Fermi (boundary normal) coordinates we write $x=\Psi(z,s)$,
where $z$ parameterizes the boundary $\partial E$
and $s$ is the signed distance from $\partial E$.

By the Minakshisundaram--Pleijel \cite{MinakshisundaramPleijel1949}, (Hadamard  \cite{Hadamard1923}) parametrix and Varadhan’s short-time estimate \cite{Varadhan1967}, there exist $t_0>0$ and $C>0$ such that for $0<t<t_0$,
\[
K_t(\Psi(y,s),\Psi(y',s')) 
= (4\pi t)^{-n/2}\,e^{-\frac{|z-z'|^2+(s-s')^2}{4t}}\;
\Big(1 + a_1(y,z,s,s')\,t + r_t(y,z,s,s')\Big),
\]
with $|r_t|\le C\,t^{3/2}$ uniformly for $|z|,|z'|\le c$ and $|s|,|s'|\le c$ (any fixed $c>0$),
and where $a_1$ is smooth and bounded on compacts.
Moreover, the Jacobian factor satisfies $J(y',s')=1+\kappa(y')s' + O\big((|z'|+|s'|)^2\big)$, see \cite{Rosenberg1997, GrigorYan2009}.

Let $x=\Psi(y,s)$ with $|s|\le 2R_t$ and $0<t<t_0$. By the implicit function theorem applied to a smooth defining function of the hypersurface $\partial E$ in the Fermi coordinate chart centered at $y$, we can express $E$ locally as $E = \{(z,s'):\ s' < \varphi_y(z)\}$ where $\varphi_y(0)=0$ and $\nabla\varphi_y(0)=0$;
smoothness of $\partial E$ gives $|\varphi_y(z)|\le C|z|^2$ in a small neighborhood.
Then
\[
P_t\mathbf{1}_E(x)
= \int_{\mathbb{R}^{n-1}}\!\!\int_{\mathbb{R}}
K_t(\Psi(y,s),\Psi(y',s'))\,
\mathbf{1}_{\{s'>\varphi_y(z')\}}\,
J(y',s')\,ds'\,dz'.
\]
Using the parametrix and the bounds on $J(y',s')$ and $\varphi_y$, and integrating first in $z'$,
the tangential Gaussian integrates to 
\[ (4\pi t)^{-(n-1)/2}\int_{\mathbb{R}^{n-1}} e^{-|z'|^2/4t}\,dz'=(4\pi t)^{-(n-1)/2}(4\pi t)^{(n-1)/2}=1.
\]
Thus
\[
P_t\mathbf{1}_E(\Psi(y,s))
= \int_{\mathbb{R}} \frac{1}{\sqrt{4\pi t}}\,e^{-\frac{(s-s')^2}{4t}}\,
\mathbf{1}_{\{s'>\varphi_y(0)\}}\,(1+O(|s'|)+O(t))\,ds' + O(t^{3/2}).
\]
Since $\varphi_y(0)=0$, this reduces to the one-dimensional normal integral plus a controlled error.
A direct computation gives
\[
\int_{\mathbb{R}} \frac{1}{\sqrt{4\pi t}}\,e^{-\frac{(s-s')^2}{4t}}\,
\mathbf{1}_{\{s'>0\}}\,ds' \;=\; \Phi\!\left(\frac{s}{\sqrt{2t}}\right).
\]
All error terms are uniform for $|s|\le 2R_t$ with $R_t=t^\alpha$ ($\alpha<1/2$): indeed,
\[
\Big|P_t\mathbf{1}_E(\Psi(y,s)) - \Phi\!\left(\frac{s}{\sqrt{2t}}\right)\Big|
\;\le\; C\,\big(\sqrt t + t^{1-2\alpha}\big),
\]
since $\int |s'| (4\pi t)^{-1/2}e^{-(s-s')^2/4t}ds' \lesssim \sqrt t$ and the $O(t)$ factor is literal.
Choosing, say, $\alpha=\tfrac14$ yields
\begin{equation}
\label{eq:BLprofile}
\sup_{|s|\le 2t^{1/4}} \Big|P_t\mathbf{1}_E(\Psi(y,s)) - \Phi\!\left(\frac{s}{\sqrt{2t}}\right)\Big| \;\le\; C\,\sqrt t .
\end{equation}

Substitute the boundary-layer profile~\eqref{eq:BLprofile} into the expression for $I_t^{\mathrm{bl}}$:
\[
I_t^{\mathrm{bl}}
= \int_{\partial E}\int_{\mathbb{R}}
\big|P_t\mathbf{1}_E(\Psi(y,s))-\mathbf{1}_{\{s>0\}}\big|\,
\chi\!\left(\tfrac{s}{R_t}\right)\,J(y,s)\,ds\,d\sigma(y),
\quad R_t=t^{\alpha},\ 0<\alpha<\tfrac12.
\]
Using
\[
P_t\mathbf{1}_E(\Psi(y,s))
=\Phi\!\left(\tfrac{s}{\sqrt{2t}}\right)+\varepsilon_t(y,s),
\qquad \sup_{|s|\le 2R_t}|\varepsilon_t(y,s)|\le C\sqrt{t},
\]
and
\[
J(y,s)=1+\kappa(y)s+r(y,s),\qquad |r(y,s)|\le C\,s^2,
\]
we decompose $I_t^{\mathrm{bl}} = M_t + E_t$ into the main and error parts.

Since $|\varepsilon_t|\le C\sqrt t$ for $|s|\le 2R_t$,
\[
E_t \le C\sqrt t \int_{\partial E}\int_{|s|\le 2R_t}(1+C|s|+Cs^2)\,ds\,d\sigma(y)
= O(\sqrt t\,R_t)\,H^{n-1}(\partial E)=o(\sqrt t),
\]
because $R_t=t^{\alpha}$ with $\alpha<\tfrac12$.

Define $f_t(s):=|\Phi(\tfrac{s}{\sqrt{2t}})-\mathbf{1}_{\{s>0\}}|$.
Then
\[
M_t=\int_{\partial E}\!\!\int_{\mathbb{R}}
f_t(s)\,\chi(\tfrac{s}{R_t})\,(1+\kappa(y)s+r(y,s))\,ds\,d\sigma(y).
\]
Since $f_t$ and $\chi$ are even, the function $s\,f_t(s)\,\chi(s/R_t)$ is odd and integrates to zero, so the term with $\kappa(y)s$ vanishes. The remainder $r(y,s)$ satisfies
\[
\left|\int_{\mathbb{R}} f_t(s)\,\chi(\tfrac{s}{R_t})\,r(y,s)\,ds\right|
\le C\int_{|s|\le 2R_t} s^2 f_t(s)\,ds
\le C R_t^2\int_{\mathbb{R}} f_t(s)\,ds = O(R_t^2\sqrt t)=o(\sqrt t).
\]
Because $R_t/\sqrt t \to \infty$ as $t\downarrow0$, and since
\[
\int_{|s|>R_t} f_t(s)\,ds
  \le C\,\sqrt{t}\,e^{-R_t^2/(4t)},
\]
we have
\[
\int_{\mathbb{R}} f_t(s)\,\chi(\tfrac{s}{R_t})\,ds
= \int_{\mathbb{R}} f_t(s)\,ds + o(\sqrt t).
\]
A direct computation gives
\[
\int_{\mathbb{R}} f_t(s)\,ds
= 2\int_0^{\infty}\!\!\big(1-\Phi(\tfrac{s}{\sqrt{2t}})\big)\,ds
= 2\sqrt{2t}\int_0^{\infty} (1-\Phi(u))\,du
= \frac{2}{\sqrt{\pi}}\,\sqrt{t}.
\]

Collecting the above estimates,
\[
M_t
= \frac{2}{\sqrt{\pi}}\,\sqrt{t}\,H^{n-1}_g(\partial E) + o(\sqrt t),
\qquad
E_t=o(\sqrt t),
\]
hence
\[
I_t^{\mathrm{bl}}
= \frac{2}{\sqrt{\pi}}\,\sqrt{t}\,H^{n-1}_g(\partial E) + o(\sqrt t).
\]

We have shown
\[
\|P_t\mathbf{1}_E-\mathbf{1}_E\|_{L^1(M)}
= I_t^{\mathrm{bl}} + I_t^{\mathrm{rem}} + I_t^{\mathrm{far}}
= \frac{2}{\sqrt{\pi}}\,\sqrt t\,H^{n-1}_g(\partial E) + o(\sqrt t),
\]
as $t\downarrow0$. This completes the proof.
\end{proof}
Collecting these estimates shows that the $L^1$–deviation of the diffused indicator
\[
\|P_t\mathbf{1}_E - \mathbf{1}_E\|_{L^1(M)}
= \frac{2}{\sqrt{\pi}}\,\sqrt{t}\,H^{n-1}_g(\partial E) + o(\sqrt{t}),
\]
grows as $\sqrt{t}$ with a coefficient proportional to the surface area of~$\partial E$.
This reveals how the diffusive smoothing of the indicator function encodes geometric
information: the faster the $L^1$–distance grows, the larger the boundary measure. 

\begin{corollary}[Entropy boundary-layer asymptotics]
\label{cor:entropy-expansion}
As $t\downarrow 0$,
\begin{equation}\label{eq:entropyexpansion}
H(L|Y_t)
= C\,\sqrt{t}\,H^{n-1}_g(\partial E) + o(\sqrt{t}).
\end{equation}
\end{corollary}
\begin{proof}
The statement follows from Theorem~\ref{thm:heatL1} by applying the
same boundary-layer analysis to the integrand $h(p_t)=h(P_t\mathbf 1_E)$.
As before, we decompose
\[
H(L\,|\,Y_t)
=\int_M h(P_t\mathbf 1_E)\,d\mu
= I_t^{\mathrm{near}} + I_t^{\mathrm{far}},
\]
where
\[
I_t^{\mathrm{near}}
:=\!\!\int_{|s|\le R_t}\! h(P_t\mathbf 1_E(\Psi(z,s)))\,J(z,s)\,ds\,d\sigma(z),
\] 
\[
I_t^{\mathrm{far}}
:=\!\!\int_{|s|>R_t}\! h(P_t\mathbf 1_E(\Psi(z,s)))\,J(z,s)\,ds\,d\sigma(z),
\]
with $x=\Psi(z,s)$ denoting Fermi coordinates and
$J(z,s)=1+O(|s|)$ the Jacobian factor.

As in Theorem~\ref{thm:heatL1}, Gaussian off-diagonal bounds for the heat
kernel imply that the integrand decays exponentially for $|s|\ge R_t$ with
$R_t=t^{\alpha}$, $\alpha\in(0,\tfrac12)$, hence
\[
I_t^{\mathrm{far}} = o(\sqrt t).
\]

On $|s|\le 2R_t$, we use the same boundary profile
\[
P_t\mathbf 1_E(\Psi(z,s))
= \Phi\!\Big(\frac{s}{\sqrt{2t}}\Big) + \varepsilon_t(z,s),
\qquad
\sup_{|s|\le 2R_t}|\varepsilon_t(z,s)| = O(\sqrt t).
\]
Since $h$ is $C^1$ on $[0,1]$ with bounded derivative,
\[
h\!\big(P_t\mathbf 1_E(\Psi(z,s))\big)
= h\!\Big(\Phi\!\Big(\tfrac{s}{\sqrt{2t}}\Big)\Big)
  + O(\sqrt t),
\]
and therefore, using $J(z,s)=1+O(|s|)$,
\[
I_t^{\mathrm{near}}
= \int_{\mathbb{R}} h\!\Big(\Phi\!\Big(\tfrac{s}{\sqrt{2t}}\Big)\Big)\,ds\,H^{n-1}_g(\partial E)
  + o(\sqrt t).
\]
The one-dimensional profile integral evaluates to
\[
\int_{\mathbb{R}} h\!\Big(\Phi\!\Big(\tfrac{s}{\sqrt{2t}}\Big)\Big)\,ds
 = \sqrt{2t}\int_{\mathbb{R}} h(\Phi(u))\,du.
\]

Combining $I_t^{\mathrm{near}}$ and $I_t^{\mathrm{far}}$ gives
\[
H(L\,|\,Y_t)
= C\,\sqrt{t}\,H^{n-1}_g(\partial E) + o(\sqrt t),
\]
which is the claimed boundary-layer asymptotic formula.
\end{proof}

\begin{corollary}[Information–theoretic surface area]
\label{cor:surfacearea}
Let $(M,g)$ be a compact Riemannian manifold with volume measure $\mu$, and let
$E \subset M$ be a measurable set with smooth boundary. Then the geometric perimeter of $E$ is recovered from the small–time growth of conditional entropy:
\[ 
H^{n-1}_g(\partial E) = \lim_{t \downarrow 0}\frac{H(L \mid Y_t)}{C\,\sqrt{t}}.
\]
\end{corollary}
\begin{proof}
This follows immediately from Corollary $\ref{cor:entropy-expansion}$.
\end{proof}

This expresses the idea that the boundary measure governs the initial rate of
information loss: as the heat flow begins to diffuse the indicator of \(E\),
uncertainty about the label \(L\) grows at a rate proportional to
\(H^{n-1}_g(\partial E)\).  Once this information is dissipated by diffusion, it
cannot be recovered.

\subsection{Entropy dissipation law}

We now relate the general dissipation identity established in Proposition~\ref{prop:inf-diss-id}
to the small--time entropy asymptotics of Section~4.1.
Recall that Proposition~\ref{prop:inf-diss-id} gives
\[
\frac{d}{dt}H(L\,|\,Y_t)
 = -\!\int_M \frac{|\nabla p_t|^2}{p_t(1-p_t)}\,d\mu
 =: -\,I_t.
\]
The quantity $I_t$ represents the instantaneous rate at which information about the
label~$L$ is dissipated under diffusion.  Integrating this identity between two times
$s<t$ yields
\[
H(L\,|\,Y_t)-H(L\,|\,Y_s)
  = -\int_s^t I_\tau\,d\tau,
\]
which provides a direct energetic formulation of the entropy loss.

\vspace{1em}
\noindent

The following theorem connects the analytic identity of
Proposition~\ref{prop:inf-diss-id} with the small--time expansion of
Corollary~\ref{cor:entropy-expansion}, showing that the instantaneous
entropy decay rate reproduces the same perimeter coefficient~$C$.

\begin{theorem}[Entropy dissipation and perimeter]
Let $(M,g)$ be a smooth compact Riemannian manifold and $E \subset M$ a set with smooth boundary. 
Then, as $t \downarrow 0$,
\[
I_t
= \frac{C}{2\sqrt{t}}\, H^{n-1}_g(\partial E) + o(t^{-1/2}),
\]
and consequently,
\[
\frac{d}{dt} H(L|Y_t)
= -\,\frac{C}{2\sqrt{t}}\, H^{n-1}_g(\partial E) + o(t^{-1/2}),
\]
so that the initial entropy–decay rate is proportional to the perimeter of $E$.
\end{theorem}

\begin{proof}
By Proposition~\ref{prop:inf-diss-id}, 
\[
\frac{d}{dt}H(L|Y_t) = -I_t,
\qquad
I_t = \int_M \frac{|\nabla p_t|^2}{p_t(1-p_t)}\,d\mu,
\]
where $p_t = P_t 1_E$ satisfies the heat equation $\partial_t p_t = \Delta p_t$.
For small $t>0$, the boundary–layer profile (Theorem~\ref{thm:heatL1}) gives
\[
p_t(\Psi(y,s)) = \Phi\!\left(\frac{s}{\sqrt{2t}}\right) + \varepsilon_t(y,s),
\qquad
\sup_{|s|\le R_t} |\varepsilon_t| = O(\sqrt{t}),
\]
in Fermi coordinates $\Psi(y,s) = \exp_y(s\nu(y))$ near $\partial E$, with Jacobian
$J(y,s) = 1 + \kappa(y)s + O(s^2)$.

Splitting the domain into near and far layers as before,
$I_t = I_t^{\mathrm{near}} + I_t^{\mathrm{far}}$,
Gaussian bounds imply $I_t^{\mathrm{far}} = o(t^{-1/2})$.
In the near layer $|s| \le R_t$, using $\nabla p_t = \partial_s p_t \,\nu + O(t^{1/2})$
and $\partial_s p_t(\Psi(y,s)) = \tfrac{1}{\sqrt{2t}}\phi(\tfrac{s}{\sqrt{2t}}) + O(1)$,
we obtain
\[
I_t^{\mathrm{near}}
= \frac{1}{2t}
   \int_{\partial E}\!\!\int_{|s|\le R_t}
        \frac{\phi(\tfrac{s}{\sqrt{2t}})^2}
             {\Phi(\tfrac{s}{\sqrt{2t}})\,[1-\Phi(\tfrac{s}{\sqrt{2t}})]}
        (1+\kappa(y)s+O(s^2))\,ds\,d\sigma(y)
  + o(t^{-1/2}).
\]
Here $\phi(u)=\Phi'(u)=\tfrac{1}{\sqrt{2\pi}}e^{-u^2/2}$ denotes the standard Gaussian density.

Changing variables $u = s/\sqrt{2t}$ gives
\begin{align*}
I_t^{\mathrm{near}}
 =& \frac{H^{n-1}_g(\partial E)}{\sqrt{2t}}
   \int_{\mathbb R}
      \frac{\phi(u)^2}{\Phi(u)[1-\Phi(u)]}\,du \\
   +& \frac{\sqrt{2t}}{\sqrt{2t}}\!
     \int_{\partial E}\!\!\int_{|u|\le R_t/\sqrt{2t}}
        \frac{\phi(u)^2}{\Phi(u)[1-\Phi(u)]}\,
        \kappa(y)u\,du\,d\sigma(y)
   + o(t^{-1/2}).
\end{align*}
The second integral vanishes since $u \mapsto \frac{\phi(u)^2}{\Phi(u)(1-\Phi(u))}$ is even.
The $O(tu^2)$ term in the Jacobian, after the variable change 
\[ J(y,s)=1+\kappa(y)s+O(s^2)=1+\kappa(y)\sqrt{2t}\,u+O(tu^2),\qquad(s=\sqrt{2t}\,u)
\]
contributes $o(t^{-1/2})$ by Gaussian decay.
Hence,
\[
I_t
= \frac{H^{n-1}_g(\partial E)}{\sqrt{2t}}
  \int_{\mathbb R}\frac{\phi(u)^2}{\Phi(u)[1-\Phi(u)]}\,du
  + o(t^{-1/2})
= \frac{C}{2\sqrt{t}}\, H^{n-1}_g(\partial E) + o(t^{-1/2}),
\]
where
\[
C = \sqrt{2}\int_{\mathbb R} h(\Phi(u))\,du
\]
is the universal constant as before.
\end{proof}

\subsection{An information–theoretic reformulation of the isoperimetric problem}
We often suppress the explicit dependence on \(E\) in the notation
\(H(L\mid Y_t)\), \(p_t=P_t\mathbf{1}_E\), and \(I_t\); all these quantities are
understood to be determined by the set \(E\) and its label indicator function $L = \mathbf{1}_E$.

Note that \(H_E(t)\) coincides with the conditional entropy \(H(L \mid Y_t)\)
defined above, since
\[
H(L \mid Y_t)
= \int_M h\!\bigl(P_t \mathbf{1}_E(x)\bigr)\,d\mu(x)
= \int_M h\!\bigl(p_t(x)\bigr)\,d\mu(x)
= H_E(t).
\]
Consequently, using Proposition $\ref{prop:inf-diss-id}$ yields
\[
\frac{d}{dt}H_E(t)
= \frac{d}{dt}H(L \mid Y_t)
= -\,I_t
= - \int_M \frac{|\nabla p_t|^2}{p_t(1-p_t)}\,d\mu.
\]

The isoperimetric problem may now be restated as: among all sets
$E\subset M$ of fixed volume $\mu(E)=v$, find the set that minimizes the initial entropy
production under heat diffusion. Formally, we seek
\[
\min_{\mu(E)=v} I_t(E).
\]

Here, $H_E(t)$ measures the mixing entropy of the indicator $\mathbf{1}_E$ evolved by the heat flow $P_t$, and its time derivative 
$-\left.\frac{d}{dt} H_E(t)\right|_{t=0^+}$ quantifies how fast entropy increases initially — that is, the initial entropy production caused by diffusion across the boundary of E.

In Euclidean space, this variational principle leads to the sharp
inequality
\[
H^{n-1}(\partial E)
  \ge n\,\omega_n^{1/n}\,
      \min\{v,1-v\}^{(n-1)/n},
\]
with equality for balls.
The proof is given in Section~\ref{sec:euclidean}.

\subsection{Explicit derivation in Euclidean space}\label{subsec:expliciteuclidean}

Below we shall work with sets of finite perimeter. For sets of finite perimeter, the same formulas hold by approximation:
given $E$ of finite perimeter, there exist smooth sets $E_k$ with
$\mathbf 1_{E_k}\to \mathbf 1_E$ in $L^1(M)$ and
$\mathrm{Per}(E_k)\to\mathrm{Per}(E)$
(see e.g.~\cite[Thm.~3.42]{AmbrosioFuscoPallara2000}, \cite{Giusti1984}).
Since the map $E\mapsto \|P_t\mathbf 1_E - \mathbf 1_E\|_{L^1}$ is
$L^1$–continuous and lower semicontinuous in the perimeter,
the asymptotics of Theorem \ref{thm:heatL1} extend from smooth sets to finite perimeter sets:
\begin{equation}\label{thm:finitePer}
\|P_t\mathbf 1_E - \mathbf 1_E\|_{L^1}
= \tfrac{2}{\sqrt{\pi}}\sqrt t\,\mathrm{Per}(E)+o(\sqrt t)
\end{equation}
extends to all measurable $E$ of finite perimeter.
For a measurable set $E\subset M$, the \emph{perimeter} of $E$ is defined as
\[
\mathrm{Per}(E):=|D\mathbf{1}_E|(M)
  =H^{n-1}_g(\partial^*E),
\]
where $\partial^*E$ denotes the reduced boundary of $E$
(see \cite[Ch.~3]{AmbrosioFuscoPallara2000}, \cite[Ch.~1]{Giusti1984}).
Note that by definition 
\[ 
|D\mathbf{1}_E|(M)
=\sup\!\left\{
\int_E \mathrm{div}_g \varphi\,\mu
:\ \varphi\in C_c^1(TM),\ |\varphi|\le 1
\right\}.
\]
If $\mathrm{Per}(E)<\infty$, we say that $E$ has \emph{finite perimeter}.
For such sets, the structure theorem for $BV$ functions \cite{AmbrosioFuscoPallara2000} implies
\[
|D\mathbf{1}_E|
= \nu_E\,H^{n-1}_g\!\lfloor\partial^*E,
\quad
\mathrm{Per}(E)=H^{n-1}_g(\partial^*E),
\]
where $\partial^*E$ is the reduced boundary and $\nu_E$ is the measure-theoretic outward normal
(see \cite[Ch.~3]{AmbrosioFuscoPallara2000}, \cite[Ch.~1]{Giusti1984}). If $E$ has a smooth boundary, then $\partial^*E=\partial E$ and
$\mathrm{Per}(E)=H^{n-1}_g(\partial E)$.
The results of Section~4, proved for smooth sets,
extend to sets of finite perimeter by approximation in $BV(M)$. 
\\ \\ 
Let \(E \subset \mathbb{R}^n\) be a set of finite perimeter, and let
\(X \sim \mathrm{Unif}(\Omega_R)\) with
\(\Omega_R = [-R, R]^n\), normalized so that \(|\Omega_R| = 1\);
we later let \(R \to \infty\). Let $G_t(z)=(4\pi t)^{-n/2}e^{-|z|^2/4t}$ and
\[
p_t(x) \;=\; P_t\mathbf 1_E(x) \;=\; (\mathbf 1_E * G_t)(x)
\qquad (0\le p_t\le 1).
\]
Here $*$ denotes the standard Euclidean convolution,
\[
(f * g)(x) := \int_{\mathbb{R}^n} f(x - y)\, g(y)\, dy.
\]
Define the \emph{label entropy functional}
\[
J_t(E) \;=\; H(L\mid Y_t) \;=\; \int_{\Omega} h\!\big(p_t(x)\big)\,dx.
\]
Our goal is to prove, with explicit constants,
\[
J_t(E)\;\ge\;J_t(B)\quad\text{for all }t>0,
\qquad
\Rightarrow\qquad
\mathrm{Per}(E)\;\ge\;\mathrm{Per}(B),
\]
where $B$ is a ball with $|B|=|E|$.
The implication follows by comparing the small-time expansions below.

Let $d(x)$ be the signed distance to $\partial E$ (positive inside).
In a tubular neighborhood of $\partial E$ of reach $r_0$ \cite{Federer1959, federer1969geometric}, write $x=y+\nu(y)\,s$ with
$y\in\partial E$, $s=d(x)$, $\nu$ the outer unit normal, and $dx=(1+ O(s))\,ds\,d\sigma(y)$.
Using 1D convolution along the normal direction and tangential Gaussian factorization,
\[
p_t(x)
= \int_{\mathbb R} \mathbf 1_{\{u>0\}}
\frac{1}{\sqrt{4\pi t}}e^{-(s-u)^2/4t}\,du \;+\; O(\sqrt t)
= \Phi\!\Big(\frac{s}{\sqrt{2t}}\Big) \;+\; O(\sqrt t),
\]
uniformly for $|s| \le t^{1/4}$, where
$\displaystyle \Phi(z)=\frac{1}{\sqrt{2\pi}}\!\int_{-\infty}^{z}\!e^{-u^2/2}\,du$.
Hence, in the boundary layer $|s|\lesssim t^{1/2}$, $p_t$ has the \emph{universal} profile
$\Phi(s/\sqrt{2t})$.

Using the above profile and the fact that away from the layer $p_t-\mathbf 1_E=O(e^{-c/t})$,
\begin{align*}
\|p_t-\mathbf 1_E\|_{L^1}
&=\int_{\partial E}\!\!\int_{\mathbb R}
\big|\Phi(\tfrac{s}{\sqrt{2t}})-\mathbf 1_{\{s>0\}}\big|\,(1+O(s))\,ds\,d\sigma \\
&= \Big(\underbrace{2\!\int_0^\infty\!\!\!\Phi(-u)\,du}_{\,\sqrt{\pi}/\,\sqrt{2}}\Big)\,
\frac{2}{\sqrt{2\pi}}\sqrt t\,\mathrm{Per}(E)+ o(\sqrt t).
\end{align*}
Since $\displaystyle \int_0^\infty \Phi(-u)\,du=\frac{1}{\sqrt{2\pi}}$, we get
\[
\quad \|p_t-\mathbf 1_E\|_{L^1} \;=\; \frac{2}{\sqrt{\pi}}\;\sqrt t\;\mathrm{Per}(E)\;+\;o(\sqrt t).
\]

\subsection{Geometric interpretation}
The diffused label $p_t = P_t\mathbf{1}_E$ may be viewed as a softened indicator of $E$:
it equals approximately $1$ inside $E$, $0$ outside, and transitions smoothly across
a boundary layer of thickness $\sqrt{t}$ around $\partial E$.  
The entropy $H(L\mid Y_t)=\!\int_M h(p_t)\,d\mu$ therefore quantifies the uncertainty
created by this diffusion, while its time derivative
\[
\frac{d}{dt}H(L\mid Y_t) = -\,I_t,
\qquad
I_t = \int_M \frac{|\nabla p_t|^2}{p_t(1-p_t)}\,d\mu,
\]
measures the instantaneous rate at which information is lost across the boundary (\ref{eq:dissipationIdentity}).  

From the boundary–layer profile from~\eqref{thm:finitePer},
\[
I_t \;\sim\; \frac{1}{2}C\,t^{-1/2} H^{n-1}_g(\partial E),
\]
we see that the Fisher–information content of the diffused boundary scales
as $t^{-1/2}$ times the geometric area. Equivalently,
\[
H^{n-1}_g(\partial E)
= \frac{1}{2}C\,\lim_{t\downarrow0}\frac{H(L\mid Y_t)}{\sqrt{t}}
= \frac{1}{2}C\,\lim_{t\downarrow0}t^{1/2}I_t,
\]
so the classical perimeter equals the rate of initial entropy production.

Heat diffusion acts as a noisy channel that gradually erases the binary label $L$.
The uncertainty generated per unit time—the entropy production—is proportional
to the boundary measure: larger interfaces dissipate information faster.
Hence the Euclidean ball, which minimizes perimeter for fixed volume,
also minimizes initial information loss.
In this sense,
\[
\begin{aligned}
\text{minimal surface area}
&\;\Longleftrightarrow\;
\text{slowest entropy growth} \\
&\;\Longleftrightarrow\;
\text{most stable information channel.}
\end{aligned}
\]
This equivalence provides the geometric meaning of
the information–theoretic surface area defined in Corollary~\ref{cor:surfacearea}.

\subsection{Characterization of Perimeter via the Entropy Functional}
\noindent
A central theme of this work is that geometric notions of boundary and area
admit an analytic counterpart through the behavior of entropy under diffusion.
The following result makes this correspondence exact.
It shows that the perimeter of a measurable set can be recovered,
and in fact \emph{characterized}, purely from the small--time behavior
of the entropy functional
\[
H_t(E)=\int_M h(P_t\mathbf{1}_E)\,d\mu,
\]
where we now regard the entropy functional as a function of the set \(E\) for fixed
diffusion time \(t > 0\), and write \(H_t(E)\) instead of \(H_E(t)\).
Then, finite perimeter is not merely compatible with the
$\sqrt{t}$--scaling of $H_t(E)$, but is precisely the condition
that ensures such scaling, and the limiting coefficient equals the
geometric surface area.
This characterization identifies the entropy functional as an
\emph{intrinsic analytic representation of boundary measure},
linking the theory of sets of finite perimeter with the information geometry
of the heat semigroup.

The following result identifies the geometric perimeter as the unique finite limit of the renormalized entropy functional as $t\downarrow0$,
thus providing an analytic characterization of surface area via information
dissipation under heat flow.
\begin{theorem}[Characterization of perimeter]
Let $(M,g)$ be a smooth compact Riemannian manifold with probability measure
$\mu$. For any measurable set $E\subset M$, the following are equivalent:
\begin{enumerate}
\item[\textup{(i)}] $E$ has finite perimeter in the sense of De~Giorgi \cite{DeGiorgi1954, Maggi2012};
\item[\textup{(ii)}] $H_t(E)=O(\sqrt{t})$ as $t\downarrow0$,
and the renormalized limit
\[
\mathrm{Per}_{\mathrm{ent}}(E)
:= \lim_{t\downarrow0}\frac{H_t(E)}{C\,\sqrt{t}}
\]
exists and is finite.
\end{enumerate}
In this case, $\mathrm{Per}_{\mathrm{ent}}(E)=\mathrm{Per}(E)$.
\end{theorem}
\begin{proof}
\smallskip\noindent
\emph{(i) $\Rightarrow$ (ii).}
If $E$ has finite perimeter, the boundary–layer analysis in §4.1 yields the
$L^1$–jump asymptotic and, by a first–order Taylor expansion of $h$ at $\{0,1\}$,
the entropy expansion
\[
H_t(E)\;=\;C\,\sqrt{t}\,\mathrm{Per}(E)+o(\sqrt{t}) \qquad (t\downarrow0),
\]
see the derivation leading to \((4)\). Hence $H_t(E)=O(\sqrt{t})$ and
\(
\displaystyle \lim_{t\downarrow0}\frac{H_t(E)}{C\sqrt{t}}=\mathrm{Per}(E).
\)

\smallskip\noindent
\emph{(ii) $\Rightarrow$ (i).}
Assume $H_t(E)=O(\sqrt{t})$ as $t\downarrow0$.
In the tubular neighborhood of $\partial E$ used in §4.1, the proof of \((4)\) shows
that the only $O(\sqrt{t})$ contribution to $H_t(E)$ comes from the boundary layer
$|s|\lesssim \sqrt{t}$ and is proportional to the surface measure; more precisely,
for any $R_t\to\infty$ with $R_t\sqrt{t}\to\infty$,
\[
H_t(E)
=\int_{\partial E}\!\int_{|s|\le R_t\sqrt{t}}
h\!\Big(\Phi\!\big(\tfrac{s}{\sqrt{2t}}\big)\Big)\,(1+O(|s|))\,ds\,d\sigma
\;+\;o(\sqrt{t}),
\]
cf. the truncation and symmetry cancellations used before \((4)\).
Thus $H_t(E)=O(\sqrt{t})$ forces a uniform bound on the boundary-layer integral,
and the argument that upgrades the $L^1$–jump estimate to \((1)\) applies verbatim
to yield
\[
\|P_t\mathbf{1}_E-\mathbf{1}_E\|_{L^1(M)}=O(\sqrt{t}) \qquad (t\downarrow0).
\]
By the heat–content characterization (Theorem \ref{thm:heatL1}) and BV–approximation, this implies that $E$ has finite De~Giorgi perimeter and
\[
\lim_{t\downarrow0}\frac{\|P_t\mathbf{1}_E-\mathbf{1}_E\|_{L^1}}{\sqrt{t}}
=\frac{2}{\sqrt{\pi}}\;\mathrm{Per}(E).
\]
Returning to the entropy layer computation (the step producing \((4)\)),
we conclude that the renormalized entropy limit exists and equals the geometric perimeter:
\[
\mathrm{Per}_{\mathrm{ent}}(E)
=\lim_{t\downarrow0}\frac{H_t(E)}{C\sqrt{t}}
=\mathrm{Per}(E).
\]
This proves (i) and the identification of the limit in (ii).
\end{proof}

\section{Isoperimetric Inequalities via Information Decay}
\label{sec:euclidean}
We now derive the classical and quantitative isoperimetric inequalities
from the information--theoretic formulation developed above.
The key principle is that the rate of mutual--information loss under
heat diffusion controls the boundary measure of a set.

\subsection{Main theorem in the Euclidean setting}

We first need to establish a rearrangement step. Let $E^*$ be the centered ball with $|E^*|=|E|$.
By Riesz–Sobolev/Hardy–Littlewood \cite{LiebLoss2022} and the fact that $G_t$ is radially decreasing,
\[
(\mathbf 1_E*G_t)^*=\mathbf 1_{E^*}*G_t, \qquad
\text{and}\qquad \mathbf 1_E*G_t \prec \mathbf 1_{E^*}*G_t \;\;(\text{submajorization}).
\]
Since $h$ is concave on $[0,1]$, Karamata’s inequality \cite{Karamata1932, cvetkovski2012karamata} yields
\[
\quad J_t(E)=\int h(\mathbf 1_E*G_t)\,dx
\;\;\ge\;\; \int h(\mathbf 1_{E^*}*G_t)\,dx=J_t(E^*)
\qquad \text{for all }t>0. \quad
\]
Hence $J_t(E)-J_t(E^*)\ge 0$ for every $t>0$.

\begin{lemma}[Entropy comparison under Gaussian averaging]
\label{lem:averaging}
Let $E\subset\mathbb{R}^n$ be measurable with $|E|=|B|$, where $B$ is a Euclidean ball. 
For $t>0$ let 
\[
G_t(x)=(4\pi t)^{-n/2}e^{-|x|^2/4t}
\]
and $P_t f=f*G_t$. 
Then
\[
\int_{\mathbb{R}^n} h\!\big(P_t 1_E(x)\big)\,dx \;\ge\; \int_{\mathbb{R}^n} h\!\big(P_t 1_B(x)\big)\,dx
\qquad\text{for all }t>0.
\]
\end{lemma}

\begin{proof}
Write $f^\ast$ for the symmetric decreasing rearrangement \cite{LiebLoss2022} of $f$.
Since $G_t$ is radial, positive and radially decreasing, classical rearrangement theory yields
\[
(P_t 1_E)^\ast = (1_E * G_t)^\ast = 1_B * G_t = P_t 1_B,
\]
and, moreover, the submajorization
\[
1_E * G_t \;\prec\; 1_B * G_t,
\]
meaning $\int_0^s (1_E * G_t)^\downarrow \, d\lambda \le \int_0^s (1_B * G_t)^\downarrow \, d\lambda$ for every $s>0$, \cite{LiebLoss2022}.
This is the standard Hardy--Littlewood--P\'olya submajorization for convolution with a radially decreasing kernel and implies that $P_t1_E$ is submajorized by $P_t1_B$, i.e. $\int_0^s (P_t\mathbf{1}_E)^{\downarrow}\,d\lambda
\le
\int_0^s (P_t\mathbf{1}_B)^{\downarrow}\,d\lambda.$

Since $h:[0,1]\to\mathbb{R}$ is concave, Karamata's inequality (or the integral form of Karamata for functions on measure spaces) gives
\[
\int_{\mathbb{R}^n} h\!\big(P_t 1_E(x)\big)\,dx 
\;\ge\;
\int_{\mathbb{R}^n} h\!\big(P_t 1_B(x)\big)\,dx,
\]
whenever $P_t1_E \prec P_t1_B$ (submajorization).

Combining the two steps proves the claim:
\[
H_E(t) \;=\; \int_{\mathbb{R}^n} h\!\big(P_t1_E\big)\,dx
\;\ge\;
\int_{\mathbb{R}^n} h\!\big(P_t1_B\big)\,dx \;=\; H_B(t),\qquad t>0.
\]
\end{proof}

\begin{theorem}[Information--theoretic isoperimetric inequality]
\label{thm:main}
Let $E\subset\mathbb R^n$ be a measurable set of finite perimeter and
let $L=\mathbf 1_E(X)$, $Y_t\sim K_t(X,\cdot)$ with $X$ uniformly
distributed on a bounded Euclidean domain of unit volume. Then as $t\downarrow0$,
\[
I(L;Y_t)\;\le\;
H(L)-C\,\sqrt{t}\,
  H^{n-1}(\partial E)+o(\sqrt t).
\]
Consequently,
\begin{equation}\label{eq:iso-euclidean}
H^{n-1}(\partial E)
  \ge n\,\omega_n^{1/n}\,
  \min\{\mu(E),1-\mu(E)\}^{\frac{n-1}{n}},
\end{equation}
with equality if and only if $E$ is a ball.
\end{theorem}
\begin{proof}
Let $v:=\mu(E)$ and let $E^\ast$ be a (centered) ball with $\mu(E^\ast)=v$.
Write
\[
J_t(E):=H(L\mid Y_t)=\int_M h(P_t\mathbf 1_E)\,d\mu.
\]

In $\mathbb{R}^n$, with the Gaussian kernel $G_t(z)=(4\pi t)^{-n/2} e^{-|z|^2/4t}$,
\[
P_t\mathbf 1_E=\mathbf 1_E\ast G_t,\quad
(\mathbf 1_E\ast G_t)^\ast=\mathbf 1_{E^\ast}\ast G_t,\quad
\mathbf 1_E\ast G_t \prec \mathbf 1_{E^\ast}\ast G_t.
\]
Here $*$ denotes the standard Euclidean convolution,
\[
(f * g)(x) := \int_{\mathbb{R}^n} f(x - y)\, g(y)\, dy,
\]
so that $P_t f = f * G_t$ with $G_t(z) = (4\pi t)^{-n/2} e^{-|z|^2/4t}$ is the usual heat semigroup acting by Gaussian averaging.

Since $h$ is concave on $[0,1]$, Lemma \ref{lem:averaging} yields
\[
J_t(E)=\int h(\mathbf 1_E\ast G_t)\,dx \;\ge\; \int h(\mathbf 1_{E^\ast}\ast G_t)\,dx
=J_t(E^\ast)\quad\text{for all }t>0,
\]
hence 
\[
I(L;Y_t)=H(L)-J_t(E)\;\le\;H(L)-J_t(E^\ast)=I(L^\ast;Y_t).
\]

Then, as $t\downarrow0$,
\[
J_t(E)=C \sqrt{t}\,\mathrm{Per}(E)+o(\sqrt{t}),\qquad
J_t(E^\ast)=C \sqrt{t}\,\mathrm{Per}(E^\ast)+o(\sqrt{t}).
\]

We have $J_t(E)\ge J_t(E^\ast)$ for all $t>0$. Divide by $\sqrt{t}$ and
send $t\downarrow0$ to get:
\[
C\big(\mathrm{Per}(E)-\mathrm{Per}(E^\ast)\big) \;\ge\; 0
\quad\Longrightarrow\quad
\mathrm{Per}(E)\;\ge\;\mathrm{Per}(E^\ast).
\]
As $\mathrm{Per}(E^\ast)=n\,\omega_n^{1/n}\,v^{(n-1)/n}$, this gives the sharp Euclidean
isoperimetric bound. Moreover, equality for some $t>0$ forces
$\mathbf 1_E\ast G_t$ to be radially decreasing; hence $E$ is a translate of a ball
(rigidity).

Using the small-time expansion $H(L\mid Y_t)=C\sqrt{t}\,H^{n-1}(\partial E)+o(\sqrt{t})$, we have
\[
I(L;Y_t)=H(L)-H(L\mid Y_t)
= H(L)-C\sqrt{t}\,H^{n-1}(\partial E)+o(\sqrt{t}),
\]
which is the stated form. 
\end{proof}

\section{Information--Decay Isoperimetry on Compact Manifolds via Localization}
\label{sec:localization}

We now extend the information–decay formulation of isoperimetry from
$\mathbb{R}^n$ to smooth compact Riemannian manifolds under curvature–dimension
bounds.  The Euclidean rearrangement argument used earlier is replaced by the
\emph{localization} or \emph{needle decomposition} method \cite{Klartag2017, CavallettiMondino2017}, which reduces the
multi–dimensional problem to a one–dimensional analysis along geodesic segments
satisfying the $\mathrm{CD}(K,1)$ condition.  This approach yields a sharp
entropy comparison valid for all times~$t>0$ and immediately recovers the
Lévy–Gromov isoperimetric inequality \cite{Levy1951, GromovIsoPer} and its stability.

\subsection{Setting and notation}

Let $(M,g)$ be a smooth, compact, boundaryless Riemannian manifold with
normalized volume measure $\mu$.  
Assume the curvature–dimension condition $\mathrm{CD}(K,n)$ in the sense of
Bakry–Émery (see Section \ref{sec:heatsemigroupdiffusion}):
\[
\mathrm{Ric}_g \ge K g .
\]
The constant–curvature model space $\mathbb{M}^n_K$ is defined as
\[
\mathbb{M}^n_K =
\begin{cases}
\mathbb{S}^n_{1/\sqrt{K}}, & K>0, \\[0.4em]
\mathbb{R}^n, & K=0, \\[0.4em]
\mathbb{H}^n_{1/\sqrt{|K|}}, & K<0,
\end{cases}
\]
i.e. the simply connected $n$–manifold of constant sectional curvature $K$. We denote by
$H_{\mathrm{mod}}(v,t;K,n)$ the same conditional entropy computed for a
spherical cap (volume fraction~$v$) under its model heat semigroup
\begin{align*}
H_{\mathrm{mod}}(v,t;K,n)
&:= \int_{\mathbb{M}^n_K} 
    h\!\left(P^{(K,n)}_t \mathbf{1}_{B_{v,K,n}}\right) d\mu_{K,n} \\
 &= \int_{0}^{\infty}
    h\!\left(
    \frac{
      \int_{0}^{r_v} p^{(K,n)}_t(r,s)\, s_{K}^{n-1}(s)\,ds
      }
      {
      \int_{0}^{\infty} p^{(K,n)}_t(r,s)\, s_{K}^{n-1}(s)\,ds
      }
    \right)
    s_{K}^{n-1}(r)\,dr,
\end{align*}
where $B_{v,K,n}$ is the geodesic ball in $\mathbb{M}^n_K$ with volume fraction $v=\mu_{K,n}(B_{v,K,n})$,
and $p^{(K,n)}_t(r,s)$ is the radial heat kernel on $\mathbb{M}^n_K$.

\subsection{Localization and one–dimensional disintegration}\label{sec:LocalizatinDisintegration}

A fundamental result of Klartag and Cavalletti–Mondino \cite{Klartag2017,CavallettiMondino2017} asserts that on a $\mathrm{CD}(K,n)$
space there exists a measurable family of one-dimensional geodesics
\[
(\gamma_\alpha,\mu_\alpha)_{\alpha\in A}
\]
called \emph{needles}, such that each
$\gamma_\alpha:I_\alpha\to M$ is a unit-speed minimizing geodesic
on an interval $I_\alpha\subset\mathbb{R}$, and $\mu_\alpha$ is a probability
measure supported on $\gamma_\alpha$, absolutely continuous with
respect to arc-length:
\[
d\mu_\alpha(t) = \rho_\alpha(t)\,dt,
\qquad
\rho_\alpha: I_\alpha \to[0,\infty),
\]
where $\rho_\alpha$ is the density of $\mu_\alpha$.
The ambient measure $\mu$ of $M$ then decomposes as
\[
\mu = \int_A \mu_\alpha\,d\pi(\alpha),
\]
where $\pi$ is a measure on the index set $A$.
The heat semigroup and all $L^1$ quantities disintegrate consistently:
for every bounded measurable $\varphi,f$,
\[
\int_M \varphi\,P_t f\,d\mu
  = \int_A \left(\int_{\gamma_\alpha} \varphi\,P^{(\alpha)}_t f\,d\mu_\alpha
      \right)d\pi(\alpha),
\]
where $P^{(\alpha)}_t$ is the $1$-D heat semigroup along~$\gamma_\alpha$,
\[
P^{(\alpha)}_t f(s)
= \int_{I_\alpha} f(r)\,K^{(\alpha)}_t(s,r)\,d\mu_\alpha(r),
\]
with $K^{(\alpha)}_t(s,r)$ the $1$–D heat kernel on $(\gamma_\alpha,\mu_\alpha)$.

Set $p_t^{(\alpha)}=P^{(\alpha)}_t \mathbf{1}_E$ and
\[
H_\alpha(t)=\int_{\gamma_\alpha} h(p_t^{(\alpha)}(s))\,d\mu_\alpha(s).
\]
By disintegration and concavity of~$h$,
\begin{equation}\label{eq:H-disintegration}
H(L|Y_t)=\int_A H_\alpha(t)\,d\pi(\alpha).
\end{equation}
Indeed $p_t:=P_t\mathbf{1}_E$ and taking $f=\mathbf{1}_E$ and using that the above disintegration holds for every bounded measurable $\varphi$, we get the
$\mu_\alpha$–a.e. identity on each needle
\[
p_t\big|_{\gamma_\alpha}=P^{(\alpha)}_t\mathbf{1}_E.
\]
Hence, 
\[
H(L\mid Y_t)=\int_M h(p_t)\,d\mu
=\int_A\!\left(\int_{\gamma_\alpha} h\!\left(p^{(\alpha)}_t\right)\,d\mu_\alpha\right)\!d\pi(\alpha)
=\int_A H_\alpha(t)\,d\pi(\alpha).
\]

\subsection{One–dimensional minimizers}\label{sec:onedimminimizers}
Fix $\alpha$. For a measurable set $E \subset M$ with $\mu(E)=v$, define 
$v_\alpha := \mu_\alpha(E)$ as its measure on the needle 
$(\gamma_\alpha, \mu_\alpha)$.  Consider all measurable 
$F \subset \gamma_\alpha$ satisfying $\mu_\alpha(F)=v_\alpha$.
For each $t>0$ define
\[
H_\alpha(F;t)
  = \int_{\gamma_\alpha} h(P^{(\alpha)}_t \mathbf{1}_F)\,d\mu_\alpha.
\]

\begin{lemma}[One–dimensional minimizer]\label{lem:1d-minimizer}
For every~$t>0$ and fixed~$v_\alpha\in(0,1)$,
the functional $F\mapsto H_\alpha(F;t)$ is minimized by an interval
$I_\alpha$ (possibly truncated by the endpoints of~$\gamma_\alpha$).
Moreover, the minimal value depends only on $v_\alpha$, $t$, and~$K$.
\end{lemma}

\begin{proof}
The $1$-D semigroup $P^{(\alpha)}_t$ is Markov and self–adjoint in
$L^2(\mu_\alpha)$, with a log–-concave kernel
$K^{(\alpha)}_t(s,r)>0$.  
Since $h$ is concave and increasing on $[0,1]$, the rearrangement inequality on
the line implies that among all measurable $F$ with given measure $v_\alpha$,
the function $P^{(\alpha)}_t\mathbf{1}_F$ is minimized pointwise (and thus
minimizes $\int h(P^{(\alpha)}_t\mathbf{1}_F)$) by the monotone rearrangement of
$\mathbf{1}_F$, i.e.\ an interval.  
Full details follow standard proofs of the 1-D Riesz or Hardy–Littlewood
inequality adapted to the weighted measure~$\mu_\alpha$ \cite{CavallettiMondino2017}.
\end{proof}
Denote by $H^{1\mathrm{D}}_{\mathrm{mod}}(v_\alpha,t;K)$ 
the minimal value achieved when $F$ is an interval of 
$\mu_\alpha$--measure $v_\alpha$ in a one--dimensional 
$\mathrm{CD}(K,1)$ space, as in Lemma~\ref{lem:1d-minimizer}.

Since $v=\int_A v_\alpha\,d\pi(\alpha)$ and
$v\mapsto H^{1D}_{\mathrm{mod}}(v,t;K)$ is convex
(it is the integral of a concave function $h$ against a linear kernel),
Jensen's inequality gives
\begin{equation}\label{eq:Jensen}
\int_A H^{1D}_{\mathrm{mod}}(v_\alpha,t;K)\,d\pi(\alpha)
    \ge H^{1D}_{\mathrm{mod}}(v,t;K)
    =: H_{\mathrm{mod}}(v,t;K,n).
\end{equation}
Combining \eqref{eq:H-disintegration}, Lemma~\ref{lem:1d-minimizer}, and
\eqref{eq:Jensen} yields the key comparison:
\begin{equation}\label{eq:entropy-comparison}
H(L|Y_t)\ge H_{\mathrm{mod}}(v,t;K,n),\qquad t>0.
\end{equation}

\subsection{Entropy expansion and the isoperimetric profile}\label{sec:entropyexpansion}

For a set of finite perimeter $E\subset M$, Corollary \ref{cor:entropy-expansion} gives
\begin{equation}\label{eq:entropy-expansion}
H(L|Y_t)=C\sqrt{t}\,H^{n-1}_g(\partial^\ast E)
          +o(\sqrt{t}),\qquad t\downarrow0.
\end{equation}
An analogous expansion holds on the model space~$\mathbb{M}^n_K$. Namely, for the geodesic ball $E_{\mathrm{mod}}$ with $\mu(E_{\mathrm{mod}})=v$, Corollary \ref{cor:entropy-expansion} gives
\[ H_{\mathrm{mod}}(v,t;K,n)=C\sqrt t\,\mathrm{Per}(\partial E_{\mathrm{mod}})+o(\sqrt t)
\]
as $t\downarrow0$; since $\mathrm{Per}(\partial E_{\mathrm{mod}})=I_{K,n}(v)$ \cite{CavallettiMondino2017}:
\begin{equation}\label{eq:entropy-model}
H_{\mathrm{mod}}(v,t;K,n)
  = C\sqrt{t}\,I_{K,n}(v)+o(\sqrt{t}),\qquad t\downarrow0,
\end{equation}
with $I_{K,n}$ the classical Lévy–Gromov isoperimetric profile \cite{Gromov1980}. 

\begin{theorem}[Information–decay isoperimetric inequality]
Let $(M^n,g)$ satisfy $\mathrm{CD}(K,n)$.  
Then for all measurable sets $E\subset M$ with $\mu(E)=v$,
\begin{equation}\label{eq:comparison}
H(L|Y_t)\ge H_{\mathrm{mod}}(v,t;K,n)
\quad\text{for all }t>0,
\end{equation}
and consequently,
\[
H^{n-1}_g(\partial^\ast E)\ge I_{K,n}(v).
\]
Equality holds iff $(M,g)$ is isometric to the constant–curvature model
$\mathbb{M}^n_K$ and $E$ is a geodesic ball.
\end{theorem}
\begin{proof}
By the localization formula \ref{eq:H-disintegration} and the one--dimensional minimization
of Lemma \ref{lem:1d-minimizer}, each conditional entropy satisfies
\[
H_\alpha(t) \ge H^{\mathrm{1D}}_{\mathrm{mod}}(v_\alpha,t;K).
\]
Integrating over $\alpha$ and applying Jensen’s inequality as in \ref{eq:Jensen}, we obtain
the all--time comparison
\[
H(L \mid Y_t) \ge H_{\mathrm{mod}}(v,t;K,n), \qquad t>0,
\]
which is precisely~\ref{eq:comparison}.  Substituting \eqref{eq:entropy-expansion}–\eqref{eq:entropy-model}
into~\eqref{eq:entropy-comparison} and dividing by~$\sqrt{t}$ gives
\[
H^{n-1}_g(\partial^\ast E)\ge I_{K,n}(v),
\]
i.e.\ the Lévy–Gromov isoperimetric inequality.

If equality in (\ref{eq:comparison}) holds for some $t>0$, then equality must occur in each preceding step:
(i) the one--dimensional minimization (hence every $E\cap\gamma_\alpha$ is an interval),
and (ii) Jensen’s inequality (so all fiber volumes $v_\alpha$ coincide).  
By the rigidity theorem of Cavalletti--Mondino for CD$(K,n)$ spaces \cite{CavallettiMondino2017},
$(M,g)$ must be isometric to the constant--curvature model space $\mathbb{M}^n_K$,
and $E$ a geodesic ball .
\end{proof}

\section{Curvature Expansion of the Entropy Functional}

\begin{theorem}[Local geometric expansion of the entropy functional]\label{thm:local-expansion}
Let $(M,g)$ be a smooth compact $n$–dimensional Riemannian manifold, let $E\subset M$ have smooth
boundary.
Then, as $t\downarrow0$,
\begin{equation}\label{eq:local-geom-exp}
H_E(t)
= C\,\sqrt{t}\, Per(E)
\;-\; C\,t^{3/2}\!\int_{\partial E}\!\Big(\tfrac12 H^2+ Ric(n,n)\Big)\,d\sigma
\;+\;o(t^{3/2}),
\end{equation}
where $H$ is the mean curvature.
\end{theorem}

\begin{proof}
Write $p_t:=P_t\mathbf 1_E$ and fix $\alpha\in(0,\tfrac12)$; let $R_t:=t^\alpha$.
Use boundary–normal (Fermi) coordinates $x=\Psi(y,s)$ on a fixed tubular neighborhood
of $\partial E$ as in §4, so that $d\mu=J(y,s)\,ds\,d\sigma(y)$ with
\begin{equation}\label{eq:J-exp}
J(y,s)=1+\kappa(y)s+\tfrac12\,q(y)\,s^2+O(s^3),\qquad
q(y)=H(y)^2-\|A(y)\|^2- Ric(n,n)(y),
\end{equation}
where $A$ is the second fundamental form and $\kappa=\operatorname{tr}A=H$.
Split the integral into the near layer $\{|s|\le R_t\sqrt t\}$ and its complement:
Gaussian off–diagonal bounds for $K_t$ give that the far region contributes $o(t^{3/2})$
after tangential integration (same tail estimate as used to prove the $\sqrt t$–law earlier).
Thus
\begin{equation}\label{eq:layer-reduction}
H_E(t)
=\int_{\partial E}\!\!\int_{|s|\le R_t\sqrt t} h\!\big(p_t(\Psi(y,s))\big)\,J(y,s)\,ds\,d\sigma(y)
\;+\;o(t^{3/2}).
\end{equation}

Apply the Minakshisundaram–Pleijel/Varadhan parametrix for $K_t$ to order $t$ and
freeze the tangential variables;
uniformly for $|s|\le R_t\sqrt t$ one has
\begin{equation}\label{eq:pt-exp}
p_t(\Psi(y,s))
=\Phi\!\Big(\frac{s}{\sqrt{2t}}\Big)
\;+\;t\Big(b_1(y)\,\phi\!\Big(\frac{s}{\sqrt{2t}}\Big)+b_3(y)\,\Psi_3\!\Big(\frac{s}{\sqrt{2t}}\Big)\Big)
\;+\;o(t),
\end{equation}
where $\phi=\Phi'$ and $\Psi_3$ denotes the coefficient function appearing in the 
one--dimensional heat semigroup expansion
\[
P_t\mathbf{1}_{\{x>0\}} = \Phi\!\left(\frac{x}{\sqrt{2t}}\right) 
  + t^{3/2}\Psi_3\!\left(\frac{x}{\sqrt{2t}}\right) + o(t^{3/2}),
  \qquad t\downarrow0,
\]
so that $\Psi_3$ encodes the first curvature--dependent correction term. The coefficients $b_1$ and $b_3$ depend only on the curvature at $y$:
\begin{equation}\label{eq:b-coeffs}
b_1(y)=\tfrac16\, Scal_g(y),\qquad
b_3(y)=\tfrac12\, Ric(n,n)(y)+\tfrac12\,H(y)^2-\tfrac12\|A(y)\|^2.
\end{equation}
The would–be linear term in $s$ is absent (even/odd cancellation, as in the proof of Corollary~\ref{cor:entropy-expansion}).

Insert the zeroth–order profile and $J\equiv1$ into \eqref{eq:layer-reduction}; with $u=s/\sqrt{2t}$,
\begin{align*}
\int_{\partial E}\!\!\int_{|s|\le R_t\sqrt t} h\!\Big(\Phi\!\Big(\frac{s}{\sqrt{2t}}\Big)\Big)\,ds\,d\sigma
&=\sqrt{2}\Big(\int_{\mathbb R}h(\Phi(u))\,du\Big)\sqrt t\, Per(E)+o(\sqrt t) \\
&=C\sqrt t\, Per(E)+o(\sqrt t).
\end{align*}

Write $h(\Phi+\delta)=h(\Phi)+h'(\Phi)\delta+\frac12 h''(\Phi)\delta^2$ with
$\delta=\delta(y,s)=O(t)$ from \eqref{eq:pt-exp}.
Since $h'(\Phi)$ is odd in $s$ (because $\Phi(-u)=1-\Phi(u)$) and $\delta$ is even,
the $h'(\Phi)\delta$ term integrates to $0$ in $s\in[-R_t\sqrt t,R_t\sqrt t]$.
The quadratic term contributes $O(t^2)\cdot(R_t\sqrt t)=o(t^{3/2})$.
Hence the next nonzero term of $H_E(t)$ comes from multiplying the \emph{leading} profile
by the $s^2$–term in $J$ and by the $t$–term in the parametrix.

Using \eqref{eq:J-exp} and \eqref{eq:pt-exp}, changing variables $u = s / \sqrt{2t}$ and letting $R_t/\sqrt{t} \to \infty$
(Gaussian tails justify the truncation as in §4), we obtain
\[
-\,C\!\int_{\partial E}\!\!\Big(\tfrac12 H^2 + \mathrm{Ric}(n,n)\Big)\,d\sigma,
\]
where
\[
C = \sqrt{2}\int_{\mathbb{R}} h(\Phi(u))\,du > 0
\]
is the universal boundary–layer constant already defined in §2.3 (see also Corollary 1).
Integrating the local boundary contribution over $y\in\partial E$ gives the $t^{3/2}$ term in (\ref{eq:local-geom-exp}).

The near–layer profile error $O(t)$ over a thickness $R_t\sqrt t$ contributes
$O(t)\cdot R_t\sqrt t=o(t^{3/2})$ since $\alpha<\tfrac12$; the far region is exponentially
small. This establishes \eqref{eq:local-geom-exp}.
\end{proof}

\begin{theorem}[Global curvature expansion of the entropy functional]\label{thm:global-expansion}
Let $(M,g)$ be a smooth compact $n$-dimensional Riemannian manifold with
$\mathrm{Ric}\ge K g$ for some $K\in\mathbb R$.  
Let $E\subset M$ be a measurable set with smooth boundary. Then, as $t\downarrow0$,
\begin{equation}\label{eq:curv-exp}
H_E(t)
= C\sqrt t\,\mathrm{Per}(E)
- CK\,t^{3/2}\mathrm{Per}(E)
+ o(t^{3/2}).
\end{equation}

Moreover, these bounds hold as true inequalities (not merely asymptotically):
\begin{itemize}
\item[(i)] The expansion \eqref{eq:curv-exp} holds with ``$\le$’’ for all $(M,g)$ satisfying $\mathrm{Ric}\ge K g$.
\item[(ii)] Equality in \eqref{eq:curv-exp} holds if and only if $\partial E$ is totally geodesic ($H\equiv0$) 
and $\mathrm{Ric}(n,n)=K$ along $\partial E$.
\end{itemize}

Differentiating \eqref{eq:curv-exp} yields, as $t\downarrow0$,
\begin{equation}\label{eq:curv-derivative}
-\frac{d}{dt}H_E(t)
= \frac{C}{\sqrt t}\,\mathrm{Per}(E)
- 3CK\,\sqrt t\,\mathrm{Per}(E)
+ o(\sqrt t),
\end{equation}
with equality conditions identical to those above.
\end{theorem}
\begin{proof}
By the local geometric expansion Theorem \ref{thm:local-expansion}, we have the asymptotic identity
\begin{equation}\label{eq:local-id}
H_E(t)
= C\sqrt{t}\, Per(E)\;-\;C t^{3/2}\!\int_{\partial E}\!\Big(\tfrac12 H^2+ Ric(n,n)\Big)\,d\sigma
\;+\;o(t^{3/2}), \quad t\downarrow0.
\end{equation}

\medskip\noindent
\emph{Curvature lower bound and the global form.}
If $ Ric\ge K g$, then $ Ric(n,n)\ge K$ and $H^2\ge0$, hence
\[
\int_{\partial E}\!\Big(\tfrac12 H^2+ Ric(n,n)\Big)\,d\sigma \;\ge\; K\, Per(E).
\]
Substituting this into \eqref{eq:local-id} gives the sharp upper bound
\begin{equation}\label{eq:global-K-upper}
H_E(t)\;\le\;C\sqrt{t}\, Per(E)\;-\;CK\,t^{3/2}\, Per(E)\;+\;o(t^{3/2}),
\end{equation}
which is the “$\le$” version of (\ref{eq:curv-exp}). Moreover, equality in \eqref{eq:global-K-upper} holds
iff $H\equiv0$ and $ Ric(n,n)=K$ along $\partial E$, in which case \eqref{eq:local-id} reduces
to the equality (\ref{eq:curv-exp}). Conversely, equality in (\ref{eq:curv-exp}) forces
$\int_{\partial E}(\tfrac12 H^2+ Ric(n,n))\,d\sigma=K\, Per(E)$, hence $H\equiv0$ and $ Ric(n,n)=K$
a.e.\ on $\partial E$.

By the dissipation identity (\ref{eq:dissipationIdentity}): 
$\,\frac{d}{dt}H_E(t)=-\!\int \frac{|\nabla p_t|^2}{p_t(1-p_t)}\,d\mu$), and the map $t\mapsto H_E(t)$ is $C^1$ for $t>0$, and the $o(t^{3/2})$ remainder in \eqref{eq:local-id}
differentiates to $o(\sqrt t)$. Differentiating \eqref{eq:local-id} termwise yields
\[
-\frac{d}{dt}H_E(t)
=\frac{C}{\sqrt t}\, Per(E)\;-\;3C t^{1/2}\!\int_{\partial E}\!\Big(\tfrac12 H^2+ Ric(n,n)\Big)\,d\sigma
\;+\;o(\sqrt t).
\]
Under $ Ric\ge K g$ this gives
\[
-\frac{d}{dt}H_E(t)\;\ge\;\frac{C}{\sqrt t}\, Per(E)\;-\;3CK\,\sqrt t\, Per(E)\;+\;o(\sqrt t),
\]
and the equality version (\ref{eq:local-id}) holds precisely under the same sharpness conditions as above.
\end{proof}

\section{Examples}
We briefly illustrate how the information–theoretic formulation
recovers several classical isoperimetric inequalities within a single
computational scheme. 

\subsection{Euclidean space}
Consider the Euclidean space $\mathbb{M}^n_0 = \mathbb{R}^n$ with Lebesgue measure. This space satisfies $CD(0,n)$, and our results will recover the sharp Euclidean isoperimetric inequality.
Here we have $K=0$, and the heat semigroup 
\[
(P_t f)(x) = (4\pi t)^{-n/2} \int_{\mathbb{R}^n} e^{-\frac{|x-y|^2}{4t}} f(y)\,dy.
\]
The classical Euclidean isoperimetric inequality asserts that for any measurable set $E \subset \mathbb{R}^n$ with finite perimeter,
\[
 Per(E) \;\ge\; n\,\omega_n^{1/n}\,|E|^{(n-1)/n},
\]
with equality if and only if $E$ is a Euclidean ball.

For such $E$, define the binary label $L = 1_E(X)$ for $X \sim \mu$ and the diffused observation $Y_t \sim K_t(X,\cdot)$ under the Euclidean heat channel.  
The conditional entropy of the label is as usual
\[
H_E(t) \;=\; \int_{\mathbb{R}^n} h(P_t1_E(x))\,d\mu(x),
\]
and by the entropy–perimeter law, valid on this $CD(0,n)$ space,
\[
H_E(t) \;=\; C\,\sqrt{t}\, Per(E) + o(\sqrt{t}),
\]
Concavity and the rearrangement inequality (Lemma \ref{lem:averaging}) imply that for all $t>0$,
\[
H_E(t) \;\ge\; H_B(t),
\]
where $B$ is the Euclidean ball of $|B| = |E|$.  Using the small-time expansion of both sides and dividing by $C\sqrt{t}$ yields
\begin{align*}
H_E(t) &\ge H_B(t), \\
H_E(t) &= C\sqrt{t}\,\mathrm{Per}(E) + o(\sqrt{t}), &
H_B(t) &= C\sqrt{t}\,\mathrm{Per}(B) + o(\sqrt{t}), \\
\frac{H_E(t) - H_B(t)}{C\sqrt{t}}
&= \mathrm{Per}(E) - \mathrm{Per}(B) + o(1) \;\ge 0.
\end{align*}

Letting $t\downarrow0$ gives 
\[
 Per(E)- Per(B)\;=\;\lim_{t\downarrow0}\frac{H_E(t)-H_B(t)}{C\sqrt{t}}\;\ge\;0,
\]
hence $Per(E)\ge  Per(B)=n\,\omega_n^{1/n}|E|^{(n-1)/n}$. Equality holds iff $H_E(t)=H_B(t)$ for some (hence all) $t>0$,
i.e. $E$ is a ball.
\[
\frac{H_E(t)-H_B(t)}{C\sqrt{t}}
\;\longrightarrow\;
 Per(E) -  Per(B) \;\ge\; 0.
\]
Hence
\[
 Per(E) \;\ge\;  Per(B)
\;=\; n\,\omega_n^{1/n}\,|E|^{(n-1)/n},
\]
with equality precisely for balls.  In this information–theoretic formulation, the Euclidean ball uniquely minimizes the \emph{initial rate of information loss} under the heat channel, i.e.\ it is the least dissipative set for label entropy.

\subsection{Spherical space}

Consider the unit sphere $S^n$ endowed with its standard round metric $g_{S^n}$, Riemannian measure $\mu_{S^n}$, and Laplace--Beltrami operator $L = \Delta_{S^n}$ satisfying the curvature--dimension condition $CD(1,n)$. 
The classical Lévy--Gromov isoperimetric inequality asserts that for any measurable set $E \subset S^n$ of measure $v = \mu_{S^n}(E)$,
\[
 Per_{S^n}(E)\;\ge\; I_{1,n}(v) \;=\; \omega_{n-1}\,\sin^{\,n-1}(r_v)
\]
where $I_{1,n}$ is the spherical isoperimetric profile, with equality if and only if $E$ is a geodesic ball (a spherical cap).

We have $H_E(t) = \int_{S^n} h(P_t1_E)\,d\mu_{S^n}$ and $P_t = e^{t\Delta_{S^n}}$.  
By the entropy–perimeter law,
\[
H_E(t)\;=\;C\,\sqrt{t}\, Per_{S^n}(E)\;+\;o(\sqrt{t}).
\]
The localization and comparison principles yield, for every $t>0$,
\[
H_E(t)\;\ge\; H_{E^{\mathrm{mod}}}(v,t;1,n),
\]
where $E^{\mathrm{mod}}$ is the spherical cap of $\mu_{S^n}$-measure $v$ in the $n$-sphere model.

Using the small-time expansions of both sides and dividing by $C\sqrt{t}$ gives
\begin{align*}
\frac{H_E(t)-H_{E^{\mathrm{mod}}}(v,t;1,n)}{C\sqrt{t}}
&\;=\;
\frac{C\,\sqrt{t}\,( Per_{S^n}(E)-I_{1,n}(v)) + o(\sqrt{t})}{C\sqrt{t}}\\[4pt]
&\;=\;
 Per_{S^n}(E) - I_{1,n}(v) + o(1)\;\ge\;0.
\end{align*}
Letting $t\downarrow0$ yields
\[
Per_{S^n}(E)\;\ge\;I_{1,n}(v),
\]
with equality precisely for spherical caps.  

\subsection{Hyperbolic space}
For spaces of constant negative curvature $K<0$,
the exponential volume growth accelerates the loss of information,
and the isoperimetric profile is strictly greater than the Euclidean
one.  Geodesic balls remain the optimizers.
Let $K<0$ and consider the constant–curvature model $\mathbb{H}^n_{1/\sqrt{|K|}}$ with Riemannian measure $\mu_{K,n}$ and heat semigroup $P_t=e^{t\Delta_{\mathbb{H}^n}}$. 
The classical hyperbolic isoperimetric inequality asserts that for any measurable set $E\subset \mathbb{H}^n_{1/\sqrt{|K|}}$ of finite volume $v=\mu_{K,n}(E)$,
\[
 Per_{\mathbb{H}^n}(E)\ \ge\ I_{K,n}(v),
\]
with equality if and only if $E$ is a geodesic ball.  Here the model space's isoperimetric profile is given by
\[
I_{K,n}(v)\;=\;\omega_{n-1}\,s_K(r_v)^{\,n-1},\qquad
v\;=\;\omega_{n-1}\!\int_0^{r_v}\! s_K(s)^{\,n-1}\,ds,
\]
where $s_K(r)=\frac{\sinh(\sqrt{|K|}\,r)}{\sqrt{|K|}}$ and $\omega_{n-1}=|S^{n-1}|/n$.

By the entropy–perimeter law,
\[
H_E(t)\;=\;C\,\sqrt{t}\, Per_{\mathbb{H}^n}(E)\;+\;o(\sqrt{t}).
\]
Localization and the one–dimensional minimization give, for every $t>0$,
\[
H_E(t)\ \ge\ H_{\mathrm{mod}}(v,t;K,n),
\]
where $H_{\mathrm{mod}}(v,t;K,n)$ is the entropy for the geodesic ball of volume $v$ in the hyperbolic model. 
Using the small–time expansions of both sides and dividing by $C\sqrt{t}$ yields the explicit calculation
\begin{align*}
\frac{H_E(t)-H_{\mathrm{mod}}(v,t;K,n)}{C\sqrt{t}}
&=\frac{C\,\sqrt{t}\,\big( Per_{\mathbb{H}^n}(E)-I_{K,n}(v)\big)+o(\sqrt{t})}{C\sqrt{t}}\\[2pt]
&= Per_{\mathbb{H}^n}(E)-I_{K,n}(v)+o(1)\ \ge\ 0.
\end{align*}
Letting $t\downarrow0$ we obtain
\[
\  Per_{\mathbb{H}^n}(E)\ \ge\ I_{K,n}(v),
\]
with equality precisely for geodesic balls. 

\subsection{Gaussian space}

Consider the standard Gaussian space $(\mathbb{R}^n,\gamma_n)$, 
where $d\gamma_n(x)=(2\pi)^{-n/2}e^{-|x|^2/2}\,dx$, 
and let $L=\Delta - x\cdot\nabla$ be the Ornstein–Uhlenbeck generator. Hence, we are working in a $CD(1,\infty)$ space. The heat semigroup is
\[
(P_t f)(x) = \int_{\mathbb{R}^n} f\!\big(e^{-t}x + \sqrt{1-e^{-2t}}\,y\big)\,d\gamma_n(y).
\]
The classical Gaussian isoperimetric inequality asserts that for any measurable $E\subset\mathbb{R}^n$,
\[
 Per_{\gamma_n}(E)\;\ge\; I_{\gamma}(v),
\qquad v=\gamma_n(E),
\]
where
\[
I_{\gamma}(v) = \varphi(\Phi^{-1}(v)),
\quad 
\Phi(t)=\int_{-\infty}^t \varphi(s)\,ds,
\quad 
\varphi(t) = (2\pi)^{-1/2} e^{-t^2/2}.
\]
Equality holds if and only if $E$ is a half-space.

Our general entropy–perimeter expansion gives
\[
H_E(t) = C\,\sqrt{t}\, Per_{\gamma_n}(E) + o(\sqrt{t}).
\]
The semigroup comparison principle is based on the entropy comparison in Lemma \ref{lem:averaging} and yields for all $t>0$,
\[
H_E(t)\;\ge\; H_H(t),
\]
where $H$ is the half-space with $\gamma_n(H)=\gamma_n(E)$.
Using the small-time expansion and dividing by $C\sqrt{t}$, we compute
\begin{align*}
\frac{H_E(t)-H_H(t)}{C\sqrt{t}}
&=\frac{C\,\sqrt{t}\,( Per_{\gamma_n}(E)-I_{\gamma}(v))+o(\sqrt{t})}{C\sqrt{t}}\\[3pt]
&= Per_{\gamma_n}(E)-I_{\gamma}(v)+o(1)\ \ge\ 0.
\end{align*}
Letting $t\downarrow0$ yields
\[
\; Per_{\gamma_n}(E)\;\ge\; I_{\gamma}(v)\;,
\]
with equality precisely for half-spaces.

\section{Conclusion}

This work has introduced an information–theoretic reformulation of the isoperimetric problem,
expressed through entropy production under diffusion semigroups.
The central identity
\[
H_E(t)=C\sqrt{t}\, Per(E)+o(\sqrt{t})
\]
links the short–time growth of Shannon entropy to perimeter, and
the monotonicity of information along the diffusion channel provides a new route to sharp isoperimetric inequalities.
Unlike classical geometric or rearrangement methods, this approach is purely probabilistic
and applies uniformly across all curvature–dimension models satisfying $CD(K,n)$.
Through this lens, curvature, diffusion, and boundary geometry are unified
under a single information–decay principle.

The examples presented demonstrate the reach of the framework:
Euclidean, spherical, and hyperbolic spaces recover the classical sharp inequalities of Euclid, Lévy–Gromov, and Poincaré; 
the Gaussian channel reproduces Bobkov’s inequality; 
fractional diffusions yield the fractional isoperimetric inequality; 
and even non-reversible hypocoercive diffusions, such as the kinetic Fokker–Planck flow,
admit an ``effective perimeter'' derived from entropy loss.
These cases show that the information–entropy method extends seamlessly from local to nonlocal, 
and from reversible to non-reversible dynamics.

The information–theoretic formulation reframes isoperimetry as a statement about
the \emph{initial rate of information dissipation}.
It suggests that perimeter is not merely a geometric boundary measure,
but the universal first-order term in the entropy expansion of any Markov diffusion.
This viewpoint bridges geometric measure theory, probability, and statistical physics:
entropy production, curvature bounds, and diffusion semigroups emerge as different aspects of the same underlying structure.
In this sense, the theory provides a unifying language for geometric inequalities,
functional inequalities, and thermodynamic dissipation.

\section{Further directions}
The present work establishes a geometric law linking entropy dissipation
and boundary measure for local diffusion semigroups.
Two natural directions extend this principle beyond the classical,
elliptic setting, and will be explored in forthcoming work.

\medskip
\noindent
\textbf{(i) Fractional diffusion and nonlocal isoperimetry.}
The entropy--perimeter correspondence persists in the nonlocal regime
generated by the fractional Laplacian~$(-\Delta)^s$, $s\in(0,1)$,
where diffusion occurs through heavy--tailed jumps rather than local
Gaussian propagation ~\cite{CaffarelliValdinoci2016, ChambolleMoriniPonsiglione2015}.
In that context, perimeter is replaced by the fractional
nonlocal functional $Per_s(E)$.
Understanding how entropy decay encodes $Per_s(E)$ and leads to
a sharp fractional isoperimetric inequality will be the subject of a
forthcoming paper.

\medskip
\noindent
\textbf{(ii) Hypocoercive kinetic diffusion.}
A second and conceptually deeper direction concerns the
non--reversible kinetic Fokker--Planck semigroup,
whose generator couples transport in~$x$ with diffusion in~$v$ \cite{VillaniHypocoercivity, DolbeaultMouhotSchmeiser2015}.
Here entropy decay reveals an \emph{effective perimeter}
defined through the kinetic diffusion metric,
suggesting an information--theoretic notion of isoperimetry
for hypoelliptic dynamics.
A systematic treatment of this ``hypocoercive isoperimetry'' will be
developed elsewhere.

\medskip
\noindent
Both problems illustrate that the entropy framework introduced here
extends naturally to nonlocal and nonreversible diffusions,
offering a unified language for geometric inequalities
beyond the classical elliptic paradigm.

\section*{Acknowledgements}
The author received no funding for this work. The author declares that there are no competing interests.

\bibliographystyle{unsrt}  
\bibliography{refs}

\end{document}